\documentclass[11pt,a4paper]{amsart}

\usepackage[all]{xy}
\usepackage[utf8]{inputenc}
\usepackage[T1]{fontenc}      
\usepackage{geometry}         
\usepackage[english]{babel}  
\usepackage{a4wide}
\usepackage{amsfonts,amsthm,amsmath,amssymb}
\usepackage{amscd,color}
\usepackage{psfrag,graphicx,subfigure,hyperref,enumerate}
\usepackage{import} 
\usepackage[shortlabels]{enumitem}
\usepackage[dvipsnames]{xcolor}
\usepackage{verbatim}
\usepackage{tikz-cd}
\usepackage[all]{xy}

\makeatletter
\@namedef{subjclassname@2020}{\textup{2020} Mathematics Subject Classification}
\makeatother

\theoremstyle{plain}
\newtheorem{theo}{Theorem}[section]
\newtheorem{thm}[theo]{Theorem}
\newtheorem{prop}[theo]{Proposition}
\newtheorem{coro}[theo]{Corollary}
\newtheorem{lem}[theo]{Lemma}
\newtheorem*{lem*}{Lemma}

\newtheorem*{thmA}{Theorem A}

\newtheorem*{corC}{Corollary C}
\newtheorem*{thmB'}{Theorem B'}
\newtheorem*{thmB}{Theorem B}

\newtheorem*{thmD}{Theorem D}
\newtheorem*{thmE}{Theorem E}

\newtheorem*{corF}{Corollary F}
\newtheorem*{corG}{Corollary G}

\theoremstyle{definition}
\newtheorem{defi}[theo]{Definition}
\newtheorem{defn}[theo]{Definition}

\newtheorem{obs}[theo]{Observation}
\newtheorem*{obs*}{Observation}
\newtheorem{rem}[theo]{Remark}

\newtheorem*{rem*}{Remark}

\numberwithin{equation}{section}

\DeclareMathOperator{\dist}{dist}

\DeclareMathOperator{\id}{Id}

\DeclareMathOperator{\w}{wind}

\newcommand{\R}{\mathbb{R}}

\newcommand{\N}{\mathbb{N}}
\newcommand{\Z}{\mathbb{Z}}
\newcommand{\C}{\mathbb{C}}

\newcommand{\D}{\mathbb{D}}

\newcommand{\chat}{\widehat{\mathbb{C}}}

\newcommand{\rs}{\hat{\mathbb{C}}}

\newcommand{\re}{\mathrm{Re}}

\newcommand{\bif}{\mathrm{Bif}}

\newcommand{\lam}{\lambda}

\newcommand{\B}{\mathbb{B}}
\newcommand{\limt}{\lim_{t \to +\infty}}
\renewcommand{\H}{\mathbb{H}}

\newcommand{\leucl}{\ell_{\rm Eucl}}

\newcommand{\g}{\color{ForestGreen}}
\definecolor{Violet}{cmyk}{0.79,0.88,0,0}
\definecolor{Lavender}{cmyk}{0,0.48,0,0}

\newcommand{\la}{\lambda}

\renewcommand{\phi}{\varphi}

\newcommand{\act}{\mathcal{A}}
\renewcommand{\AA}{\mathcal{A}}
\newcommand{\ra}{\rightarrow}
\newcommand{\FF}{\mathcal{F}}
\newcommand{\JJ}{\mathcal{J}}

\newcommand{\acal}{\mathcal{A}}

\title{Bifurcation loci of families of finite type meromorphic maps}
\date{\today}

\author{Matthieu Astorg$^{\dag}$}
\author{Anna Miriam benini$^*$}
\author{N\'uria Fagella$^\ddag$}

\thanks{$^{\dag}$ Partially supported by the grant ANR JCJC Fatou ANR-17- CE40-0002-01 and the grant ANR PADAWAN 21-CE-40-0012-01}
\thanks{$^*$  Partially supported by the Indam group GNAMPA, and by  PRIN 2017, Real and Complex Manifolds: Topology, Geometry and holomorphic dynamics.}
\thanks{$^\ddag$  Partially supported by grants PID2020-118281GB-C32 and CEX2020-001084-M (Maria de Maeztu Excellence program) from the Spanish state research agency,  and 
ICREA Acad\`emia 2020 from the Catalan government.   }
\address{ M. Astorg:  Institut Denis Poisson\\ Université d'Orléans\\
  France} \email{matthieu.astorg@univ-orleans.fr}
\address{ A.M. Benini: Dipartimento  di   Matematica Fisica e Informatica\\
Universit\'a di Parma\\ Italy
} \email{ambenini@gmail.com}
\address{ N.  Fagella$^{1,2}$:  \newline 
$1$ Dep. de Matem\`atiques i Inform\`atica\\ Universitat de Barcelona\\ Barcelona\\  Spain \newline
$2$ Centre de Recerca Matem\`atica\\ Barcelona\\ Spain} \email{nfagella@ub.edu}


\begin{document}
\maketitle

\begin{abstract}
We show that $\JJ-$ stability is open and dense in natural families of finite type meromorphic maps, that is, meromorphic maps of one complex variable with a finite number of singular values. This extends the results of Ma\~{n}\'e-Sad-Sullivan \cite{mss} for rational maps of the Riemann sphere and those of Eremenko and Lyubich \cite{el} for entire maps of finite type of the complex plane.
This result is obtained as a consequence of a detailed study of a new type of bifurcation that arises with the presence  poles in addition to  essential singularities (namely periodic orbits exiting the domain of definition of the map along a parameter curve), and in particular  its relation with the bifurcations in the dynamics of singular values.   
The presence of these new bifurcation parameters require essentially different methods to those used in previous work for rational or entire maps. 

%
%
\end{abstract}
\section{Introduction}

Structural stability is a key concept in dynamical systems which is attributed to Andronov and Pontryagin in 1937 and even further to Poincaré in the early 1880's. Roughly speaking, a dynamical system is structurally stable (also called {\em robust}) if its qualitative properties do not change under small perturbations. Somewhat more precisely, a discrete dynamical system $f:X\to X$ (in a certain class of regularity $\mathcal C$) is structurally stable, if all maps $g$ sufficiently close to $f$ in $\mathcal C$ are topologically conjugate to $f$, that is, if there exist a homeomorphism $\varphi:X\to X$ such that $\varphi \circ f =g\circ \varphi$, and $\phi$ depends continuously on $g$.

The problem of density of structurally stable systems in the appropriate class is subtle 
and has been around for a long time.  While this property 
is generically true for $\mathcal C^r$ maps, $r\geq 1$, in real dimension one \cite{KSV}, works of Smale, Williams and Newhouse among others (see e.g. \cite{smale66, CS}) concluded that structurally stable systems are not dense in the space of diffeomorphisms on manifolds of real dimension larger or equal than 2.
%
%

In the context of holomorphic dynamics, a more natural but closely related notion is \emph{$\JJ$-stability}, which informally speaking means structural stability in restriction to the Julia set (see Section  \ref{sec:holo} for a precise definition). The seminal work of Ma\~n\'e, Sad and Sullivan \cite{mss} and Lyubich \cite{lyu1} showed that $\JJ$-stable systems form an open and dense set in the space of holomorphic maps on the Riemann sphere (i.e. rational maps).   Their results therefore solved the problem of density of $\JJ-$stability for holomorphic maps on {\em compact} Riemann surfaces.

One of the key factors in the proof is the renowned $\lambda-$lemma, tying $\JJ$-stability to the holomorphic movement 
of periodic orbits which, in a compact manifold, can only fail when periodic orbits collide in a parabolic cycle. With this tool in hand, they provided a complete set of equivalences between possible notions of stability in parameter space:  $\JJ$-stability,  stability of critical orbits and 
stability of periodic cycles. The subset of functions where these properties fail to hold is called the \emph{bifurcation locus}.
The study of bifurcation loci is related to some of the deepest questions in holomorphic dynamics: for instance, the bifurcation locus of the quadratic family $\{z \mapsto z^2+\lam\}_{\lam \in \C}$  is exactly the boundary of the famous Mandelbrot set, whose fine description is the object of the principal conjecture in the field (MLC conjecture).

When trying to extend these results to maps on non-compact manifolds (like the complex plane) one must deal with an additional possibility for the failure of periodic orbits being analytically continued, namely the possibility of periodic orbits {\em exiting the domain} at a certain parameter value. As an example, one can observe this new type of bifurcation occurring at $\la_0=0$ for the family $f_\la(z)=z+\la+e^z$ , where the fixed points {\em disappear to infinity} (the essential singularity) when considering curves of parameters converging to $0$.  Eremenko and Lyubich \cite{el} showed that this phenomenon does {\em not} occur for  entire maps of the complex plane with a finite number of {\em singular values} (points where some branch of the inverse fails to be well-defined), known as {\em finite type} entire maps. Consequently, they were able to conclude that  $\JJ-$stability is also dense in this class of functions.


In the presence of poles, that is, in the context of finite type meromorphic maps, simple examples like  $\lambda \tan z$ for $\la_0=\pi/2$, show that this new type of bifurcations of cycles disappearing to infinity {\em do} occur and hence  obstruct most of the arguments used for the Stability Theorem in \cite{mss,lyu1}. 
 
The goal of this paper is to prove that  $\JJ$-stability is dense also in this setting, and we accomplish it by performing a detailed analysis of the new type of bifurcation. In particular we will see how these bifurcation parameters relate to the stability of singular orbits (Theorems A and B) and to parabolic parameters (Theorem D), resulting in a Stability Theorem (Theorem E), from which we conclude that the bifurcation locus has empty interior (corollary F).


One may wonder whether our results might extend beyond finite type meromorphic maps.
 A simple example (see Section \ref{bifdense}) shows that,  in general,  density of $\JJ$-stability fails for natural  families  of entire maps  which are not of finite type. Other   results in \cite{EpsRem} (unpublished) provide an example of a family of maps with an infinite (but bounded) set of singular values (hence not of finite type) and infinitely many attractors, in the spirit of the Newhouse example \cite{newhouse}. 
 Nevertheless, results analogous to the ones presented here are likely to hold  for general finite type maps in the sense of Epstein (holomorphic finite type maps defined on open subsets of the complex plane) and will be investigated in a subsequent paper.
  Finally, considering  possible generalizations to higher dimensions, structural stability (in the sense of Berteloot, Bianchi and Dupont \cite{BBD}) is known {\em not} to be dense in the family of endomorphisms of $\mathbb{P}^k$ for any $k\geq 2$ (see e.g. \cite{dujardin, taflin,biebler}).

\subsection*{\large \bf Statement of results} 

We start by giving some necessary definitions, starting by the holomorphic families which are the object of our study.

\begin{defi}[Natural family] \label{def:natfam}
	Let $M$ be a complex connected manifold. A {\em natural family} of finite type meromorphic maps is a  family $\{f_\lam\}_{\lam \in M}$ of the form $f_\lam=\phi_{\lam}\circ f_{\lam_0} \circ \psi_{\lam}^{-1}$, where $f:=f_{\la_0}$ is a finite type meromorphic map, and
	$\phi_\lam, \psi_{\lam}$ are quasiconformal homeomorphisms depending holomorphically on $\lam \in M$, and such that $\psi_{\lam}(\infty)=\infty$.
\end{defi}

Under these conditions, one can check that $f_\la$ depends holomorphically on $\la$, (i.e. $\la \mapsto f_\la(z)$ is holomorphic for every fixed $z\in \C$). 

Let $S(f)$ denote the set of singular values (critical or asymptotic, see Section \ref{subsec:asympv}) of a meromorphic map $f$.  If $f$ is of finite type, then it can be embedded in a  finite dimensional complex analytic manifold of dimension $\# S(f)+2$ \cite{el,EpsteinThesis}, obtained by allowing $\phi$ and $\psi$ to be any pair of homeomorphisms that make $\phi\circ f\circ \psi^{-1}$ holomorphic (or meromorphic) and satisfy $\psi(\infty)=\infty$. Hence natural families can be viewed as subfamilies in this natural parameter space.

Many simple families are natural, like for example the exponential $E_\la(z)=\la  e^z$, the tangent $T_\la(z)=\la \tan(z)$ or the quadratic family $Q_\la(z)=z^2+\la$. In these three examples, the map  $\phi_\la$ is conformal, and $\psi_\la=\id$. Notice that the singular values of $f_\la$ are marked points that can be followed holomorphically in $\la\in M$, hence their number and their nature (critical or asymptotic) remain constant throughout the entire family. The same is true for their preimages: both critical points and {\em asymptotic tracts} (logarithmic preimages of punctured neighborhoods of the asymptotic values, (see Section \ref{subsec:asympv})  can be followed holomorphically in $\la$ and their multiplicity remains constant. One may naturally ask how restrictive is the concept of a natural family.  As we show in Theorem \ref{thm:natural}, the answer is that the properties described above are necessary and sufficient conditions for an arbitrary holomorphic family of maps to be locally natural. Hence, since all of our main results are local in parameter space, 
they still apply if we replace the assumption that $\{f_\lam\}_{\lam \in M}$ is natural by the  assumption that  $S(f_\lam)$ and $f_\lam^{-1}(S(f_\lam))$ both move holomorphically over $M$.

Next we define the concept of a cycle disappearing to infinity or exiting the domain (both terms will be used indistinctively).
\begin{defi}[Cycle disappearing to infinity]\label{def:exit}
	Let $\{f_\lam\}_{\lam \in M}$ be a holomorphic family of meromorphic maps.
	We say that a cycle of period $m \geq 1$ {\em disappears to infinity} at $\lam_0 \in M$  (or {\em exits de domain} at $\la_0$)  if there exist two continuous curves $t \mapsto \lam(t)$ and 
	$t \mapsto z(t)$ such that:
	\begin{enumerate}
		\item for all $t>0$, $\lam(t) \in M$ and $z(t) \in \C$, with $f_{\lam(t)}^m(z(t))=z(t)$
		\item $\limt \lam(t)=\lam_0 \in M$ and $\limt z(t)=\infty$.
	\end{enumerate}
\end{defi}
As mentioned above, Eremenko and Lyubich \cite{el} showed that cycles cannot exit the domain for holomorphic families of entire functions.

The phenomenon of cycles disappearing to infinity was previously observed in several particular slices of meromorphic functions starting with the early studies of the tangent family $T_\la(z)=\la\tan(z)$ by Devaney, Keen and Kotus \cite{KK, DevaneyKeen, DK}, see also \cite{kotus2002diffusion}, followed by several other families with two asymptotic values  \cite{CK19,CKY22} and generalized to some dynamically defined one-dimensional families 
in \cite{FK}.  Following the terminology in the literature we define virtual cycles which, as we will see, describe limits of actual cycles after they disappear at infinity.

\begin{defi}[Virtual cycle]\label{def:virtualcycle}
	Let $f: \C \to \chat$ be a meromorphic map. 
	A {\em virtual cycle} of length $n$ is a finite, cyclically ordered set $z_0, z_1, \ldots, z_{n-1}$   such that for all $i$, either $z_i \in \C$ and $z_{i+1}=f(z_i)$,
	or $z_i=\infty$ and $z_{i+1}$ is an asymptotic value for $f$, and at least one of the $z_i$ is equal to $\infty$.  If $z_i=\infty$ only for one value of $i$ then we say that the virtual cycle has {\em minimal length} $n$.
\end{defi}
	
 If a virtual cycle remains after perturbation within the family, then it is called {\em persistent}.

\begin{defi}[Persistent virtual cycle]\label{def:persistentvc}
	Let $\{f_\lam\}_{\lam \in M}$ be a holomorphic family of meromorphic maps,
	let $\lam_0 \in M$ and assume that $f_{\lam_0}$ has a virtual cycle
	$z_0, \ldots, z_{n-1}$.
	If there exist holomorphic germs $\lam \mapsto z_i(\lam)$ defined near $\lam_0$ in $M$
	such that
	\begin{enumerate}
		\item $z_i(\lam_0)=z_i$
		\item  $z_i(\lam)\equiv\infty$ if $z_i=\infty$
		\item and $z_0(\lam), \ldots, z_{n-1}(\lam)$ is a virtual cycle for $f_\lam$,
	\end{enumerate}
 	then we say that $z_0, \ldots, z_{n-1}$ is a \emph{persistent} virtual cycle.
\end{defi}

In particular in a holomorphic family, if $v(\lam_0)$ is an asymptotic value such that $f_{\lam_0}^n(v(\lam_0))=\infty$ for some $n \geq 0$, then  $(v(\lam_0), f_{\lam_0}(v(\lam_0)), \ldots, \infty)$ is a virtual cycle of  minimal length $n+1$. (The case $n=0$ corresponds to the situation where $\infty$ itself is an asymptotic value).  This  virtual cycle is persistent if and only if the singular relation  $f_{\lam}^n(v(\lam))=\infty$ is satisfied in all of $M$.  If this is not the case, i.e. if a virtual cycle for $\la_0$ is non-persistent, we will say that  $\la_0$ is  a {\em virtual cycle parameter}.

Our next and last definition concerns the concept of activity or passivity of a singular value
(see \cite{Levin}, \cite{mcmbook}).

\begin{defi}[Passive (active) singular value]\label{defn:active singular values}
	Let $\{f_\lam\}_{\lam \in M}$ be a holomorphic family of finite type 
	rational, entire or meromorphic maps. Let $v(\lam)$ be a singular value (or a critical point) of $f_\lam$ 
	depending holomorphically on $\lam$ near some $\lam_0 \in M$.
	We say that $v(\lam)$ is \emph{passive} at $\lam_0$ if there exists a neighborhood $V$ of $\lam_0$ in $M$ such that:
	\begin{enumerate}
		\item	either $f_\lam^n(v(\lam))=\infty$ for all $\lam \in V$; or
		\item the family
		$\{\lam \mapsto f_\lam^n(v(\lam)) \}_{n \in \N}$ is well-defined and normal on $V$.
	\end{enumerate}
We say that  $v(\la)$ is {\em active} at $\la_0$ if it is not passive. 
\end{defi}

We are now ready to state our first result, which connects the three concepts defined above: cycles disappearing to infinity, virtual cycles and the activity of a singular value.

\begin{thmA}[Activity Theorem]
		Let $\{f_\lam\}_{\lam \in M}$ be a natural family of finite type meromorphic maps, and
		 assume  that a cycle of period $n$ disappears to infinity at $\lam_0 \in M$. Then, this cycle converges to a virtual cycle for $f_{\lam_0}$,  which contains (at least) either an active \emph{asymptotic value}, or an active \emph{critical point}. 
	\end{thmA}

Note that by definition, activity means here that there exists parameters 
arbitrarily close to $\lam_0$ for which  one of the critical points or asymptotic values in the virtual cycle does not remain in the backward orbit of $\infty$.

Let us observe how Theorem A implies that cycles cannot disappear at infinity in the finite type entire setting, hence recovering the main theorem \cite[Theorem 2]{el}. Indeed, because of the lack of poles, it is easy to see that if a cycle of period $n$ disappears at infinity, then every point of the cycle must converge to infinity (and not just one). This means that the limit virtual cycle is of the form $\infty, \ldots, \infty$.
In particular, it does not contain any critical points; Theorem A then asserts that $\infty$ itself must be an active asymptotic value for $f_{\lam_0}$. But this is impossible, since for families of finite type entire maps $\infty$ is always a \emph{passive}
asymptotic value. 

Finally, we remark that if the virtual cycle contains an active asymptotic value, then this virtual cycle is non-persistent  (as defined in Definition \ref{def:virtualcycle}). It would be interesting to know if there are examples of such limit virtual cycles in which all asymptotic values are passive.
\medskip

We now state our second result, which in a sense is a converse to Theorem A.

\begin{thmB}[Accessibility Theorem]
		Let $\{f_\lam\}_{\lam \in M}$ be a natural family of finite type meromorphic maps, and $\lam_0 \in M$ be 
		such that $f_{\lam_0}$ has a non-persistent virtual cycle  of minimal length $n+1$
\[
v(\lam_0), f_{\lam_0}(v(\lam_0)), \ldots, f^n(v(\lam_0))=\infty.
\]
		Then there is a cycle of period $n+1$ exiting the domain at $\lam_0$. Moreover, this cycle can be chosen so that its multiplier goes to zero
		as it disappears to infinity.
	\end{thmB}

In particular, by the definition of a virtual cycle, $v(\lam_0)$ is an asymptotic value (finite or infinite) and hence $\la_0$ is a virtual cycle parameter. In the terminology of \cite{FK}, the Accessibility Theorem  states that every virtual cycle parameter is also a {\em virtual center} (since the multiplier of the disappearing cycle is tending to $0$ at the limit parameter), and it is  accessible from the interior of a component in parameter space for which the analytic continuation of this cycle remains attracting. This proves the main conjecture in \cite{FK} in much greater generality than originally stated.
 
Putting together Theorems A and B, we obtain the following immediate corollary.

\begin{corC}
	Let $\{f_\lam\}_{\lam \in M}$ be a natural family of finite type meromorphic maps, and assume that 
	this family does not have any persistent virtual cycle. Let $\lam_0 \in M$.
	Then a cycle disappears to infinity at $\lam_0$ if and only if
	$f_{\lam_0}^n(v(\lam_0))=\infty$ for some asymptotic value $v(\lam_0)$.
\end{corC}

Up to this point we have described the new type of bifurcation that occurs in the presence of poles and asymptotic values. Observe that this phenomenon is {\em a priori} unrelated to the collision of periodic orbits forming parabolic cycles, in contrast to what occurs for rational or entire maps for which this is the only possible obstruction for the holomorphic motion of the Julia set. Our next result shows that,  nevertheless,  when an attracting cycle disappears at some parameter value, this can be approximated by parabolic parameters.

\begin{thmD}[Approximation by parabolic parameters]
	Let $\{f_\lam\}_{\lam \in M}$  be a natural family of finite type
	meromorphic maps, and assume that an attracting cycle of period $n$ disappears to infinity at $\lam_0 \in M$. Then there exists a
	sequence $\lam_k \to \lam_0$ such that $f_{\lam_k}$ has a non-persistent parabolic cycle
	of period at most $n$.
\end{thmD}

In particular, by Theorem B, this happens at every virtual cycle parameter. Moreover, using other approximation results in Section \ref{sect:bifloc}, it follows from Theorem D that parabolic parameters are actually dense in the activity locus (Corollary \ref{cor:density parabolic parameters}).

We now state our last main result,  which is an extension of Mañe-Sad-Sullivan's and Lyubich's bifurcation theory in the setting of finite type meromorphic maps. We stress here the fact that the proof relies in a crucial way on both Theorems A and B.

\begin{thmE}[Characterizations of $\JJ$-stability]
		Let $\{f_\lam\}_{\lam \in M}$ be a natural family of finite type meromorphic maps.
		Let $U \subset M$ be a simply connected domain of parameters. The following are equivalent:
		\begin{enumerate}[\rm (1)]
				\item The Julia set moves holomorphically over $U$ (i.e. $f_\la$ is $\JJ$-stable for all $\la\in U$)
				\item Every singular value is passive on $U$.
				\item  The maximal period of attracting cycles is bounded on $U$.
				\item  The number of attracting cycles is constant in $U$.
				\item  For all $\la\in U$, $f_\la$ has no non-persistent parabolic cycles.
		\end{enumerate}
	\end{thmE}

\begin{figure}[hbt!]
\includegraphics[width=0.7\textwidth]{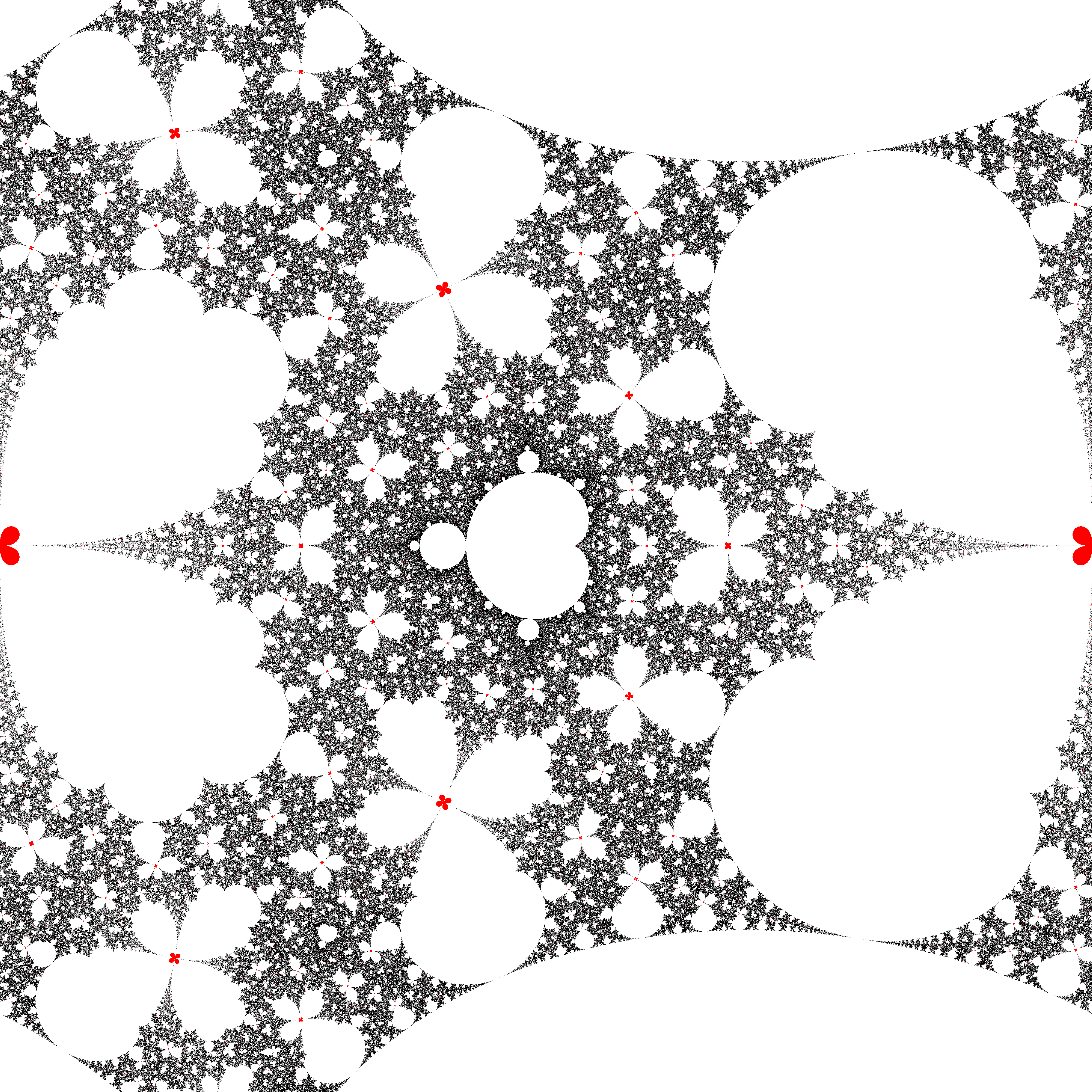}
\caption{\small \label{fig:sp} Bifurcation locus of the natural family $f_\lam(z)=\pi \tan^2(z)+\lam$.
These maps have only one critical value and one asymptotic value ($\lam$ and $\lam-\pi$ respectively), which are then both mapped to the same point. Centers of hyperbolic components and Misiurewicz parameters are shown in black. Parameters for which there exist an attracting cycle are shown in white.  Virtual cycle parameters are at the center of the red crosses.}
\end{figure}

In view of Theorem E it makes sense to define the {\em bifurcation locus} of the natural family as 
\[
\bif(M)=\{\la \in M\mid f_\la \text{ is not $\JJ$-stable} \},
\]
or equivalently as the set of parameters for which some of the conditions in Theorem E is not satisfied. Since $\JJ-stable$ parameters form an open set by definition, following the arguments in \cite{mss} we obtain the  following statement, so well known for rational maps. 

\begin{corF}[$\JJ$- stable parameters form an open and dense set in $M$] 
If $\{f_\lam\}_{\lam \in M}$ is a natural family of finite type meromorphic maps, then $\bif(M)$ has no interior or, equivalently, $\JJ-$stable parameters are open and dense in $M$. 
\end{corF}

Our last Corollary gives meaning to the notion of hyperbolic components in our setting.

\begin{corG}
	If $\{f_\lam\}_{\lam \in M}$ is a natural family of finite type meromorphic maps, and $U$ is a connected component of $M \setminus \bif(M)$ containing a parameter $\lam_0$ such that all singular values of $f_{\lam_0}$ are captured by attracting cycles, then the same holds for every $\lam \in U$. 
\end{corG}

%
%
%
%

\subsection*{Structure of the paper} In Section~\ref{sec:prel} we introduce the concepts of tracts, holomorphic and natural families, holomorphic motions and  $\JJ$-stability. We then  show that holomorphic motions of singular values and their preimages imply that a holomorphic family is (locally) natural (Theorem \ref{thm:natural}). We also prove a 'shooting lemma' (Proposition~\ref{pseudomontel}) which replaces the use of Montel's Theorem when activity of a  singular value is due to truncation of its orbits. 
 Section \ref{sec:exiting} connects cycles disappearing to infinity to  virtual cycles, and the latter to active singular values, proving Theorem A. Section \ref{sect:Accessibility} shows that virtual cycles are the  limits of attracting cycles tending to infinity on which the  multiplier tends to zero, making them 'virtual centers' and proving Theorem B.  Section \ref{sect:bifloc}  contains theorems about the density of several types of parameters in the bifurcation locus of natural families, including the proof of Theorem D.  Additionally,   all the elements are assembled to  prove Theorem E, the meromorphic analogue of Ma\~n\'e-Sad-Sullivan and Lyubich results, and its corollaries.  Finally, Section \ref{bifdense} 
gives a simple example showing that, without the finite type assumption, the bifurcation locus may be the entire complex plane. 
 

\subsection*{Acknowledgements} 

We would like to thank Lasse Rempe for helpful discussions, and especially for pointing out to us the statement and proof of  Lemma~\ref{lem:lasse}. We are also grateful to Sarah Koch, Dylan Thurston and Lukas Geyer for interesting conversations and specially to Curt McMullen for his thoughtful comments and suggestions.
  Finally we are grateful to Bob Devaney, Linda Keen and Janina Kotus for motivating discussions and for their introduction to this beautiful subject.


\section{Preliminaries}\label{sec:prel}

In this section we state some known results and prove several new tools that will be useful in the proofs of the main theorems.

\subsection{Dynamics, asymptotic values and logarithmic tracts}\label{subsec:asympv}
\ 

 Given a rational or entire function $f$, the {\em Fatou set} $\FF(f)$ or {\em stable set} of $f$ is defined as the largest open set where the family of iterates $\{f^n\}_{n\geq 0}$ is normal in the sense of Montel. However, if $f:\C \to \chat=\C\cup\{\infty\}$ is a meromorphic (transcendental) map with at least one non-omitted pole, we need to require additionally that the family of iterates $\{f^n\}_{n\geq0}$  is first well defined {\em and} then normal.  In both cases, the {\em Julia set} is the complement of the Fatou set and it is  the closure of the repelling periodic points. If the backwards orbit of infinity $\mathcal{O}^-({\infty})$ is an infinite set, the Julia set also coincides with its closure, i.e.  $\JJ(f)=\overline{\mathcal{O}^-({\infty}})$, or equivalently the closure of the set of prepoles of all orders.  For background on iteration theory of meromorphic maps we refer to the survey \cite{bergweiler} and references therein.

The dynamics of $f$ are determined to a large extent by the dynamics of its  {\em singular values} or points  $v\in\chat$ for which not all univalent branches of $f^{-1}$ are locally well defined. If $f$ is rational, singular values are always {\em critical values} $v=f(c)$, being $c$ a {\em critical point} (i.e. $f'(c)=0$). If $f$ is transcendental we must also take {\em asymptotic values} into account, that is values $v=\lim_{t\to\infty} f(\gamma(t))$ where $\gamma$ is a curve such that $\gamma(t)\to \infty$ 
as $t\to\infty$. An example is $v=0$ for the exponential map.  

We say that $f$ is  of {\em finite type}, if it has a finite number of singular values, forming the set
\[
S(f)=\{ v\in \chat\mid \text{$v$ is a critical or an asymptotic value of $f$}\}.
\]
Maps of finite type possess specific dynamical properties which are not satisfied in the  general cases. For one, their Fatou set is made exclusively of preperiodic or periodic components, the latter being basins of attraction of attracting or parabolic orbits, or rotation domains (Siegel disks or Herman rings).

In the terminology of \cite{BE08}, asymptotic values always have transcendental singularities lying over them. More precisely, we have the following definition.

\begin{defi}[Tracts]\label{defi:tracts}
	Let $f: \C \to \hat{\C}$ be a meromorphic map,  let $v \in \hat{\C}$ be an asymptotic value of $f$, and $D$ be a punctured disk centered at $v$ of radius $r>0$. We say that a simply connected unbounded set $T$ is a \emph{logarithmic tract above $v$} if $f:T\ra D $ is an  infinite degree unbranched covering.
\end{defi}
Notice that by decreasing the value of the radius $r$, we obtain a nested sequence of logarithmic tracts shrinking to infinity. We say that two logarithmic tracts (or sequences thereof) are different, if they are disjoint for sufficiently small values of $r$. The number of different logarithmic tracts lying over an asymptotic value $v$ is the {\em multiplicity} of $v$.

If $v$ is an isolated asymptotic value one can see that there is always a logarithmic tract above $v$ (see e.g. \cite{BE08}). In particular, if $f$ is of finite type all of its asymptotic values have logarithmic tracts lying over them. For this reason, in this paper we will call them simply {\em tracts}.

\subsection{Holomorphic families, holomorphic motions and $\JJ$-stability} \label{sec:holo}    

We give here the precise definitions of what it means to follow the Julia sets holomorphically, given a family of meromorphic maps.

\begin{defn}[Holomorphic family] \label{def:holofam}
A {\em holomorphic family $\{f_\la\}_{\la\in M}$ of meromorphic maps} over a complex connected manifold $M$  is a  holomorphic map $F:M \times \C \longrightarrow \chat$ such that $F(\la, \cdot )=:f_\lam$  is a non-constant meromorphic map for every $\la\in M$.  
\end{defn}

\begin{defn}[Holomorphic motion]
A {\em holomorphic motion} of a set $X \subset \hat\C$ over a set $U\subset M$ with basepoint $\la_0\in U $ is a map  $H: U  \times X \rightarrow \chat$ given by 
 $(\la,x) \mapsto H_\la(x) $ such that
 \begin{enumerate}
\item  for each $x \in X$ , $H_\la(x)$ is holomorphic in $\la$, 
 \item   for each  $\la \in U$,  $H_\la(x) $ is an injective function of  $x \in X$, and,
  \item  at $\la_0$, $H_ {\la_0} \equiv {\rm Id}$.
  \end{enumerate}

A holomorphic motion of a set $X$ {\em respects the dynamics} of the holomorphic  family $\{f_\la\}_{\la\in M}$ if 
 $H_\la \circ f_{\la_0} = f_\la \circ H_{\la}$ whenever both $x$ and $f_{\la_0}(x)$ belong to $X$.
\end{defn}

Note that the continuity of $H$ is not required in the definition. However this property follows as a  consequence, as shown in the $\la-$Lemma proved in \cite{mss}.

\begin{thm}[The $\la-$Lemma \cite{mss}]\label{lem:la-lemma}
A holomorphic motion $H$ of $X$ as above has a unique extension to a holomorphic motion of $\overline{X}$. The extended map $H: U \times \overline{X} \to \chat$ is continuous, and for each $\la \in U$, $H_\la: \overline{X} \to \chat$ is quasiconformal. Moreover, if $H$ respects the dynamics, so does its extension to $\overline{X}$.
\end{thm}

 
  For further results about holomorphic motion and the $\la-$Lemma, see for instance  \cite{ast}. 

\begin{defn}[$\JJ$-stability]\label{def:jstability}
Consider as above a holomorphic family $F: M \times \C \to \chat$ of meromorphic maps. Given $\la_0\in M$, the map $f_{\la_0}$  is $\JJ$-stable if there exists a neighbourhood $U\subset M$ of $\la_0$ over which the Julia sets move holomorphically, i.e.~there exits a holomorphic motion $H: U\times \JJ_{\la_0} \to \chat$ such that $\H_\la (\JJ_{\la_0}) = \JJ_\la$ and furthermore it respects the dynamics. 
\end{defn}

By virtue of the $\la$-Lemma and the density of periodic points in the Julia set, it is enough to construct a holomorphic motion of the set of periodic points of every period, to obtain one for the entire Julia set .

\subsection{ Natural families}
\label{sec:natural}

The goal of this section is to prove Theorem \ref{thm:natural} below.
%
%
But let us first begin by making a few important observations. Let $f=f_{\la_0}$ for some  $\la_0\in M$ and let $\{f_\la=\phi_\lam \circ f \circ \psi_\lambda^{-1}\}_{\lam \in M}$ be a natural family, as in Definition \ref{def:natfam}. Then  $\psi_\lam$ maps the critical points of $f_{\lam_0}$ to those of $f_\lam$, and 
$\phi_\lam$ maps the critical values and asymptotic values of $f_{\lam_0}$ to those of $f_{\lam}$.
In particular,  as observed in the Introduction, in a natural family, the critical points and singular values always move holomorphically with the parameter 
and can never collide, while the multiplicity of each singular value remains constant throughout the family. A converse of this statement, provided by the following theorem,  is also true, though  not immediate.

\begin{thm}[Characterization of natural families] \label{thm:natural}
	Let $\{f_\lam\}_{\lam \in M}$ be a holomorphic family of finite type meromorphic maps, on which
	$S(f_\lam)$ and $f_\lam^{-1}(S(f_\lam))$ both move holomorphically. Then for every $\lam \in M$, there is a neighborhood $V$ of $\lam$ such that $\{f_\lam\}_{\lam \in V}$ is a natural family.
\end{thm}

The proof of  Theorem \ref{thm:natural} requires the following  lemma.

\begin{lem}\label{lem:familylift}
Let $X, Y$  be two connected and  path connected topological spaces. Suppose that $\{g_t:Y\to X\}_{t\in[0,1]}$,  is a homotopy such that $g_t$ is a covering for every $t\in [0,1]$.
 Then, there exists a homotopy $\{h_t: Y \to Y\}_t $ such that every $h_t$ is a homeomorphism satisfying
\[
g_t =  g_0 \circ h_t.
\] 
\end{lem}
\begin{proof}
By the homotopy lifting property, 
since the covering map $g_0$ can be lifted by the identity map, there exists a unique homotopy $\{h_t:Y \to Y\}_{t\in[0,1]}$ that lifts $g_t$, hence satisfying the following diagram
\[
\begin{tikzcd}[row sep=10ex, column sep=10ex]
&  Y \arrow[d,"g_0"] \\
Y \arrow[r, "g_0"] \arrow[r, bend right, "g_t"'] 
\arrow[ur,"id"'] \arrow[ur, bend left, "h_t"] & X.
\end{tikzcd}
\]
It remains to prove that $h_t$ is a homeomorphism for every $t$. 

We first observe that  $h_t$ is a local homeomorphism since $g_t$ is a local homeomorphism (and a covering) for every $t$ and the diagram commutes. Next, since  $h_t$ is a homotopy to the identity map, we have that $(h_t)_*(\pi_1(Y,y_0)) = (\id)_*((\pi_1(Y,y_0))) = \pi_1(Y, y_0)$, from which it follows that $h_t$ must also be a homeomorphism for every $t$.
\end{proof}

\begin{proof}[Proof of Theorem~\ref{thm:natural}]
	Let $\lam_0 \in M$, and $f:=f_{\lam_0}$. 
	We let $Y:=\C \backslash f^{-1}(S(f))$ 
	and $X:=\chat \backslash S(f)$. Let $B \subset M$ be a small ball centered at $\lam_0$.
	By Bers-Royden's Harmonic $\lam$-lemma \cite{bers1986holomorphic},	
	we can extend the holomorphic motion of $S(f_\lambda)$ over $B$ to
	a holomorphic motion $(\lam,z) \mapsto \phi_{\lam}(z)$ of the whole Riemann sphere, 
	over the ball $V$ of radius one third that of $B$.
	Similarly, we extend the holomorphic motion of $f_\lam^{-1}(S(f_\lam))$
	to a holomorphic motion $(\lam,z) \mapsto \chi_\lam(z)$. Notice that  both maps are continuous as maps of $V\times \chat\to\chat$, holomorphic in $\la$ and quasiconformal in $z$.

	Consider the family of  maps $g_\lam:=\phi_\lam^{-1} \circ f_\lam \circ \chi_\lam$ (in particular, $g_{\lam_0}=f$). If we could prove that $g_\la\equiv f$, the proof would be complete, since $f_\la$ would then be of the required form. This is not quite true, but we will show that in fact
 $g_\lam = f \circ h_\lam$, for some  family of homeomorphisms $h_\lam$, which will still be enough to conclude.

Observe that the maps $g_\lam:\C\setminus f^{-1}(S(f)) \to \chat\setminus S(f)$ are covering maps, in general  only continuous (both in $z$ and in $\lambda$ as well as jointly) and not holomorphic.  Our first claim is that  there exists a continuous family of homeomorphisms $\{h_\la:\C\setminus f^{-1}(S(f)) \to \C\setminus f^{-1}(S(f)) \}_{\la \in M}$ such that
\[
\psi_\la:= \chi_\la \circ h_\la^{-1}: \C\setminus f^{-1}(S(f)) \to \C\setminus f_\lam^{-1}(S(f_\lam))
\]
depend holomorphically on $\la$ and the following diagram commutes:
\[
\begin{tikzcd}[row sep=10ex, column sep=10ex]
\C\setminus f^{-1}(S(f)) \arrow[dr, "g_{\la_0}=f"] \arrow[dd,bend right=50, "\psi_\la"']&  \\
\C\setminus f^{-1}(S(f)) \arrow[u, "h_\la"] \arrow[r, "g_\la"]  \arrow[d,"\chi_\la"']
	& \hat{\C}\setminus S(f) \arrow[d,"\varphi_\la"]\\
\C\setminus f_\lam^{-1}(S(f_\lam)) \arrow[r, "f_\la"] & \chat \setminus S(f_\lam)
\end{tikzcd}.
\]

To define $h_\la$,  let us fix $\la\in V$ and consider a path $\Lambda:[0,1]\to V$ joining $\la_0$ and $\la$. By Lemma \ref{lem:familylift}, there exists a homotopy $\{\tilde{h}_t: \C\setminus f^{-1}(S(f)) \to \C\setminus f^{-1}(S(f))\}_t$ composed of homeomorphisms such that 
\[
g_{\Lambda(t)} = f \circ \tilde{h}_t.
\]
We then define $h_\la:=\tilde{h}_1$. Since $V$ is simply connected, the base points are irrelevant and hence $(\la, z )\to (\la, h_\la(z))$ is continuous.
 
 Since both $h_\la^{-1}$ and $\chi_\la$ are homeomorphisms, we have proven that $\psi_\la$ is a homemorphism for every $\la\in M$. It is left to prove that for every for every $z \in \C\setminus f^{-1}(S(f))$, $\lam \mapsto \psi_{\lam}(z)$ is holomorphic in $\lambda$.

 To see this, observe that we have $\phi_{\lam} \circ f(z)=f_\lam \circ \psi_\lam(z)$, and that, fixing $z \in \C\setminus f^{-1}(S(f))$,  the map
	$F(\lam, y)= \phi_{\lam} \circ f(z) - f_\lam(y)$ is holomorphic. Fix $\la_1\in V$ and  $y_1=\psi_{\la_1}(z)$. Then, by construction $f_{\la_1}'(y_1)\neq 0$  and hence, by the Implicit Function Theorem, there exists a map $\lam \mapsto y(\lambda)=\psi_\la(z)$ such that $y_1=y(\la_1)$ and $F(\la, y(\la))=\varphi_\la(f(z))-f_\la(\psi_\la(z))=0$. 
It follows that $\lam \mapsto \psi_{\lam}(z)$ is holomorphic locally near every $\lam_1 \in V$, and so is holomorphic on $M$.

	Finally, we have proved that
	\begin{enumerate}
		\item for each $\lam \in V$, $\psi_\lam: Y \to \C \backslash  f_\lam^{-1}(S(f_\lam))$ is injective
		\item for every $z \in Y$, $\lam \mapsto \psi_{\lam}(z)$ is holomorphic on $Y$.
	\end{enumerate}
	In other words, 
	$(\lam,z) \mapsto \psi_\lam(z)$ is a holomorphic motion 
	of $\C \backslash f_\lam^{-1}(S(f_\lam))$. By applying again Bers-Royden's Harmonic $\lam$-lemma and replacing again $V$ by a ball of smaller radius,
	$\psi_\la$  extends  to a holomorphic motion of $\chat$ (fixing $\infty$), such that for every $\lam \in V$, 
	$\psi_{\lam}: \chat \to \chat$ is a quasiconformal homeomorphism.
	
Since $f_\la \circ \psi_\la = \varphi_\la \circ f$, this proves that $\{f_\lam\}_{\lam \in V}$ is a natural family.
 \end{proof}

Before we end this section, we record here a lemma which illustrates the convenience of working with natural families.

\begin{lem}\label{lem:infinitypassive}
	Let $\{f_\lam\}_{\lam \in M}$ be a natural family of finite type meromorphic maps, such that for all $\lam \in M$, $f_\lam$ has finitely many poles. Then the number of poles is independent of $\lam$. Moreover, for all $\lam \in M$, $\infty$ is an asymptotic value for $f_\lam$. 
\end{lem}

\begin{proof}
	Let us write $f_\lam = \phi_\lam \circ f \circ \psi_\lam^{-1}$, with $\psi_\lam(\infty)=\infty$ and $f=f_{\lam_0}$ for some $\lam_0 \in M$, and $\phi_{\lam}, \psi_{\lam}$ are both quasiconformal homeomorphisms. Then the set of poles of $f_\lam$ is the image under  $\psi_\lam$ of the set $f^{-1}( \{\phi_\lam^{-1}(\infty)\})$. In particular, this set  $f^{-1}( \{\phi_\lam^{-1}(\infty)\})$ is finite by assumption, so $\phi_\lam^{-1}(\infty)$ is  
 a  Picard  exceptional value for $f$. By Picard's theorem and connectivity of $M$, the continuous map $\lam \mapsto \phi_\lam^{-1}(\infty)$ is constant: $\phi_{\lam}^{-1}(\infty)\equiv x \in \rs$.
	
	Therefore the set of poles of $f_\lam$ is given by $\psi_{\lam} \left( f^{-1}(\{x\})   \right)$, which is a holomorphic motion of $f^{-1}(\{x\})$. In particular, its cardinality is independent of $\lam$.
	
	For the second assertion, it suffices to observe that if $f_\lam$ has finitely many poles, then $\infty$ is a Picard exceptional value, and therefore an asymptotic value \cite[Chapter 5, Theorem 1.1]{golost}.
\end{proof}


\subsection{Quasiconformal distortion}

We state here a well-known distortion estimate for quasiconformal homeomorphisms  that we will need in the  proof of the main theorems.

\begin{lem}[Distortion of small disks]\label{lem:distortion}
	Let $\{\varphi_\lam\}_{\lam \in \D}$ be a holomorphic motion of the Riemann sphere $\rs$, with $\varphi_0 = \id$.
	Let $t \mapsto \lam(t)$ be a continuous path in $\D$ with $\lim_{t \to +\infty} \lam(t)=0$, and $t \mapsto r_t$ a continuous function with $r_t >0$ and 
	$\lim_{t \to +\infty} r_t=0$. Let $t \mapsto z_t$ be a path in $\rs$ and $D_t:=\D(z_t,r_t)$. Let $\epsilon>0$;
	then for all $t$ large enough:
	\begin{equation*}
		\D(\varphi_{\lam(t)}(z_t), r_t^{1+\epsilon}) \subset \varphi_{\lam(t)}(\D(z_t,r_t)) \subset \D(\varphi_{\lam(t)}(z_t), r_t^{1-\epsilon})
	\end{equation*}
\end{lem}

\begin{proof}
	By Theorem 12.6.3 p. 313 in \cite{astala2008elliptic}, for all $t>0$, $\theta \in \R$ and  and $r \leq 1$, we have : 
	$$|\varphi_{\lam(t)}(z_t+r e^{i\theta}) - \varphi_{\lam(t)}(z_t) | \leq e^{5(K_{\lam(t)}-1)} \cdot |\varphi_{\lam(t)}(z_t)- \varphi_{\lam(t)}(z_t + e^{i\theta}) | \cdot r^{1/K_{\lam(t)}},$$
	where $K_\lam>1$ is the dilatation of $\varphi_\lam$.
	Since  $\varphi_{\lam(t)} \to  \id$ uniformly on $\rs$ as $t \to +\infty$, we have that as $t \to +\infty$:
	$$\sup_{\theta \in [0,2\pi]} |\varphi_{\lam(t)}(z_t)- \varphi_{\lam(t)}(z_t + e^{i\theta}) | \to 1.$$
	Since $(\varphi_\lam)$ is a holomorphic motion,  $K_\lam \to 1$ as $\lam \to 0$ and so 
	$K_{\lam(t)} \to 1$ as $t \to +\infty$.
	The inclusion  $\varphi_{\lam(t)}(\D(z_t,r_t)) \subset \D(\varphi_{\lam(t)}(z_t), r_t^{1-\epsilon})$
	then follows.

	The other inclusion is equivalent to $\varphi_{\lam(t)}^{-1}(\D(y_t,r_t^{1+\epsilon})) \subset \D(\varphi_{\lam(t)}^{-1}(y_t), r_t)$, with $y_t:=\phi_{\lam(t)}(z_t)$. Its proof is essentially the same and 
	is left to the reader. 
\end{proof}

We also record here the following well-known property of quasiconformal mappings, that we will need in the proof of Theorem A.

\begin{lem}[See \cite{el}, Lemma 4]\label{lem:ELLemma4}
	Let $\psi:\hat{\C}\ra\hat{\C}$ be a $K$-quasiconformal homeomorphism fixing $0,\infty$.  Let $\arg\psi(z)-\arg z$ be a uniform branch of the difference of arguments in $\C^*$. Suppose 
	\[
	B^{-1}\leq|\psi(z_0)|\leq B, \hspace{20pt} |\arg\psi(z_0)-\arg z_0|\leq B,
	\]
	for some  $z_0\in\C$ and  $B>0$.
	Then for $|z|>|z_0|$ the following estimates hold: 
	\begin{align}
		\label{eqtn:ELLem4Eq1}&C^{-1}|z|^{K_1^{-1}}\leq|\psi(z)|\leq C|z|^{K_1}\\
		\label{eqtn:ELLem4Eq2}&|\arg \psi(z)-\arg z|\leq K_1 \ln|z|+C.
	\end{align}
	Here $K_1, C$ depend on $K,z_0,B$ but not on $z,\psi$. 
\end{lem}

Finally, we close this preliminary subsection with the following lemma, which we will use frequently:

\begin{lem}\label{lem:Gopendiscrete}
	Let $D \times \chat : (\lam,z) \mapsto \psi_\lam(z)$ be a holomorphic motion of $\chat$ over a domain $D \subset \chat$, where the $\psi_{\lam}: \chat \to \chat$ are quasiconformal homeomorphisms. Let 
	$g: D \to \chat$ be a non-constant holomorphic function, and assume that there exists $\lam_0 \in D$ such that for all $\lam \in D$,
	$\psi_\lam(g(\lam_0))=g(\lam_0)$. Let $G(\lam):=\psi_\lam^{-1} \circ g(\lam)$.
	Then there is a neighborhood $V \subset D$ of $\lam_0$ such $G_{|V} : V \to \chat$ is open and discrete.
\end{lem}

\begin{proof}
	We first observe that $G(\lam)=y \iff \psi_{\lam}(y)-g(\lam)=0$.
	Let $F_y(\lam):=\psi_{\lam}(y)-g(\lam)$: then $F_y$ is a holomorphic map, which depends continuously on $y \in \rs$. Moreover, it follows from the assumption on $\lam_0$ that $F_{y_0}$ is non-constant, where $y_0:=g(\lam_0)$. Therefore, there is an open neighborhood $W$ of $g(\lam_0)$ such that for all $y \in W$, $F_y$ is non-constant. We let $V:=G^{-1}(W)$, which is open since $G$ is continuous.
	
	Let us first prove that $G$ is open. Let $\lam_1 \in V$ and $y_1:=G(\lam_1)$. 
	Then by the previous observation, $F_{y_1}(\lam_1)=0$. Moreover, since $y \mapsto F_y$ is continuous, Hurwitz's theorem implies that all $y$ close enough to $y_1$, $F_y$ has a zero $\lam$ (close to $\lam_1$); in other words, $G(\lam)=y$. Therefore $G$ is indeed open.
	
	Moreover, $G^{-1}(\{y\})$ is the set of zeroes of $F_y$, so it is discrete.
\end{proof}

\begin{rem}\label{rem:Gbranched}
	By a result due to Stoilow \cite{stoilow1932proprietes}, Lemma \ref{lem:Gopendiscrete} 
	implies that there exists homeomorphisms $h_1$ and $h_2$ defined over respective neighborhoods of $\lam_0$ and $G(\lam_0)$ such that $h_2 \circ G \circ h_1^{-1}(\lam)=\lam^m$, for some $m \geq 1$.
	In other words, $G$ is locally a finite degree (possibly) branched cover. In particular, it is always possible to lift curves in a neighborhood of $G(\lam_0)$, although the lift is a priori not unique if $m \geq 2$.
\end{rem}

Although we do not require it, it is possible to prove by standard arguments that the map $G$ is in fact quasiregular.

\subsection{A shooting Lemma} \label{sect:shoot}

In Section \ref{sec:density} we shall prove that certain sets of parameters are dense  in the bifurcation locus. To that end, we will need the fact that,  if some singular value $v(\la_0)$ is mapped to infinity in finitely many steps (i.e. $f_{\lam_0}^n(v(\lam_0))=\infty$ for $n \geq 0$ and $f_{\lam}^n(v(\lam))\not \equiv \infty$ on $M$),  then we can find nearby parameters for which $v(\lam)$ has some prescribed dynamical behaviour. Similar results can be proven in the rational setting using Montel's Theorem together with the non-normality of the family of iterates of the active singular value. 
In our setting in which $f: \C \to \hat{\C}$ is a transcendental meromorphic map, and $U \subset \C$ is a domain, the singular value $v_\lambda$ could be active because its family of iterates $\{f_\lambda(v_\lambda)\}_{n\in\N}$ is not defined in a parameter neighborhood of $\lambda_0$ rather than not being normal. As a consequence, one cannot always apply Montel's Theorem  as for entire maps or rational maps. Its role will be played by the following statement, which holds for any natural family of meromorphic maps. Notice that here we do not have assumptions on the set of singular values so that a priori functions could be in  the general class of meromorphic transcendental functions.  Nevertheless, observe that the Proposition is only meaningful when  there exists at least one non-omitted pole.

\begin{prop}[Shooting Lemma]\label{pseudomontel}
Let $\{f_\la\}_{\la\in M}$ be a natural family of meromorphic maps.\ 
Let $\la_0\in M$  be such that a singular value $v_\la$ satisfies $f_{\lam_0}^n(v_{\la_0})=\infty$,  but this relation is not satisfied for all $\la\in M$. 
 Let $\lambda \mapsto \gamma(\la)$ be a  holomorphic map  such that $\gamma(\la_0) \notin S(f_{\la_0})$. Then we can find $\la'$ arbitrarily close to $\la_0$  such that
$f_{\lam'}^{n+1}(v_{\la'}) = \gamma(\la')$.
\end{prop}
 Observe that the map $\gamma$, which is a holomorphic map from a neighborhood of $\lambda_0$ in $M$ to $\hat{\C}$, is allowed to be constant. In particular, if $\infty$ is not a singular value of $f_{\la_0}$, by taking $\gamma(\la)\equiv \infty$ we obtain that the  parameter $\la_0$ is a limit of parameters $\la_k$ for which $v(\la_k)$ is a prepole of order $n+1$.

\begin{rem*}
A weaker version of Proposition \ref{pseudomontel} (allowing $n+2$ instead of $n+1$) could  be proven by showing that the map $\la \mapsto f_\la^{n+1}(v(\la))$ has an essential singularity at $\la_0$, and applying the Great Picard's Theorem. This was pointed out to us by R. Roeder.
\end{rem*}

The proof of Proposition~\ref{pseudomontel} uses the following lemmas. The first  one can be found in  \cite[Lemma 13]{bfjk18} (see also \cite[Lemma 4.6]{benfag} for a more general statement). In the following, let us denote by $\w(\sigma(t),P)$ the winding number of a curve $\sigma(t)$ with respect to a point $P$.

\begin{lem}[Computing winding numbers] \label{homolemma}
Let   $\gamma, \sigma:[0,1] \to \C$ be  two disjoint closed curves and let $P_\gamma \in \gamma$ and $P_\sigma \in \sigma$ be arbitrary points.  Then
 \begin{equation}\label{lem:wind}
 \w(\sigma(t)-\gamma(t),0)=\w(\gamma(t),P_\sigma) + \w(\sigma(t),P_\gamma).
 \end{equation}
\end{lem}
As a consequence, we obtain the following. 

\begin{lem}[Fixed point theorem] \label{shooting} Let $V$ be a Jordan domain, and let $f,g$ be holomorphic functions in a neighborhood of $\overline{V}$. Suppose that  $g(\overline{V})\subset f(V)$ and  $g(\partial V)\cap f(\partial V)=\emptyset$. Then there exists $\lambda\in V$ such that $f(\lambda)=g(\lambda)$.
\end{lem}

\begin{proof}
Consider the map $F(\la)=f(\la)-g(\la)$. Let $\la(t), t\in[0,1]$ be a parametrization of $\partial V$, and notice that $f(\la(t))$ and $g(\la(t))$ are two disjoint curves and hence $F(\la(t))\neq 0$ for every $t\in[0,1]$. By the Argument Principle, if the winding number of $F(\la(t))$ with respect to 0 is positive, then $F$ has at least one zero in $V$. 

Let $P_f=f(\la(0))$ and $P_g=g(\la(0))$. Applying Lemma \ref{homolemma} we get
\[
\w(F(\la(t)),0)=\w(f(\la(t))-g(\la(t)),0)= \w(g(\lambda(t)), P_f) + \w(f(\lambda(t)), P_g).
\] 
The hypothesis $g(\overline{V})\subset f(V)$ implies that the curve $g(\lambda(t))$ lies inside a bounded connected component of the complement of $f(\lambda(t))$ from which we deduce that $ \w(g(\lambda(t)), P_f)=0$. The same hypothesis also implies that $P_g\in f(V)$ which means, again by the Argument Principle,  that $\w(f(\la(t))- P_g),0) = \w(f(\la(t)), P_g) \geq 1$.
Hence $\w(F(\la(t)),0)>0$ and the conclusion follows.
\end{proof}

We will also need the following well known fact.

\begin{lem}[Shrinking of holomorphic images]\label{shrink}
 Let $U\subset\C$ be an open set and $a,b\in\C$. Suppose $\{\phi_n:U\to\C\setminus \{a,b\}\}_{n \in \N}$ is a sequence of holomorphic maps such that $\phi_{n}(u_0)\to \infty$  for a certain $u_0\in U$. Then for every compact set $K\subset U$, the spherical diameter of $\phi_{n}(K)$ tends to 0. 
\end{lem}
\begin{proof}
We claim that $\{\phi_n\}_{n \in \N}$ converges locally uniformly to $\infty$.
By Montel's Theorem, the sequence $\{\phi_n\}_{n \in \N}$ admits converging subsequences.  Let $\{\phi_{n_k}\}_{k \in \N}$ be any such  subsequence, and let $\phi: U \to \hat{\C}$ be the limit function.
Since by assumption for all $k \in \N$, $\infty \notin \phi_{n_k}(U)$ and $\phi(u_0)=\infty$, it follows from Hurwitz's Theorem that $\phi \equiv \infty$. Since this holds for any converging subsequence, we have $\lim_{n \rightarrow \infty} \phi_n=\infty$, and the lemma follows.
\end{proof}
%


Observe that if the singular value   $v_{\la_0}$
 is a prepole of order $n$,
then the map $\la \mapsto f_\la^n(v_\la)$ is a well defined meromorphic map in a sufficiently small neighborhood of $\la_0$, with an isolated pole at $\la_0$. Indeed, if a sequence of such parameters of order equal to $n$ were to accumulate at $\la_0$, by the discreteness of zeros of holomorphic functions we would have that  $\la\mapsto f_\la^k(v_\la)$ is identically equal to $\infty$ for some $k\leq n$, which contradicts the assumption that this is not a persistent condition. Also impossible would be an approximating sequence of  likewise parameters of order strictly less than $n$ since, by continuity, the order of $\la_0$ would also need to be strictly less than $n$, also a contradiction. 
 As a consequence of this fact, $\la\mapsto f_\la^{n+1}(v_\la)$ has an essential singularity at $\la_0$. 

We are now ready to prove Proposition~\ref{pseudomontel}.
\begin{proof}[Proof of Proposition \ref{pseudomontel}]

First, we pick an arbitrary one-dimensional slice containing $\lam_0$ in the parameter space $M$ on which $\lam \mapsto f_\lam^n(v(\lam))$ is not constant, and we identify $M$ with $\D(\lam_0,1) \subset \C$ in the rest of the proof.

By assumption $f_\la=\phi_\la\circ f\circ \psi_\la^{-1}$ and we may assume without loss of generality that $\phi_{\la_0}=\psi_{\la_0}={\rm Id}$ and hence $f=f_{\la_0}$. Let $D$ be a disk centered at $\gamma(\la_0)$ such that $\overline{D}$ is disjoint from $S(f_{\la_0})$ and let $\delta>0$ be such that $\gamma(\overline{\D(\la_0,\delta)}) \subset D$ (see Figure~\ref{fig:pseudomontel}).

Decreasing $\delta$ if necessary, the function $G(\la):=\psi_\la^{-1} ( f_\la^n(v_\la))$ satisfies the assumptions of Lemma \ref{lem:Gopendiscrete} (and of Remark \ref{rem:Gbranched}) on $\D(\la_0,\delta)$,
with $g(\lam):= f_\la^n(v_\la)$ and $g(\lam_0)=\infty$.
 It follows that $G(\D(\la_0,\delta))$ contains a disk of spherical radius say $\epsilon>0$ centered at $\infty$.  

Since there are no singular values in $\overline{D}$ and $f_{\la_0}$ has infinite degree, there are infinitely many univalent preimages of $D$ under $f_{\la_0}$  which must accumulate at infinity. Observe that these preimages must miss, for example, a given periodic orbit of period 3 which does not intersect $D$. Hence, selecting a subset of those preimages if necessary, we may assume (see Lemma \ref{shrink}) that they are all  bounded and that in fact their spherical diameter tends to 0. Let $U$ be one such preimage contained in $\D_s(\infty,\epsilon)$. Thus $f_{\la_0}(U)=D$. 

Since $U$ belongs to the image of $G$, we let $V$ denote a connected component of $G^{-1}(U)$ inside $\D(\la_0,\delta)$. If $D$
(and therefore $U$) is small enough, then $V$ is a Jordan domain as well.
Let us now define $F(\la):= f_\la^{n+1}(v_\la)$. Our goal is to show that $\overline{\gamma(V)} \subset F(V) $ so that Lemma \ref{shooting} applied to $\gamma$ and $F$ gives the result. 

In order to see this we write
\[
f_\la^{n+1}(v_\la) = \phi_\la \circ f_{\la_0} \circ \psi_\la^{-1} \circ f_\la^n (v_\la) = \phi_\la \circ f_{\la_0} \circ G(\la),
 \]
and therefore
\[
F (V) = \phi_\la ( f_{\la_0} (G(V)))= \phi_\la(f_{\la_0}(U))=\phi_\la(D).
\]
Now since $\delta$ can be taken arbitrarily small, the values of $\la$ can be arbitrarily close to $\la_0$ and therefore $\phi_\la$ is arbitrarily close to the identity. It follows that $F (V)=\phi_\la(D) \simeq D$, while $\gamma(V) \subset \gamma(\overline{\D(\la_0,\delta)}) \subset D$. { Moreover, $\partial \gamma(\overline{\D(\la_0,\delta)}) $ separates the boundaries of these two sets, so the hypotheses of Lemma \ref{shooting} can be applied and we are done.}

\begin{figure}[hbt!]
\begin{center}
\def\svgwidth{\textwidth}
\begingroup%
  \makeatletter%
    \setlength{\unitlength}{\svgwidth}%
  \begin{picture}(1,0.4129605)%
    \put(0,0){\includegraphics[width=\unitlength]{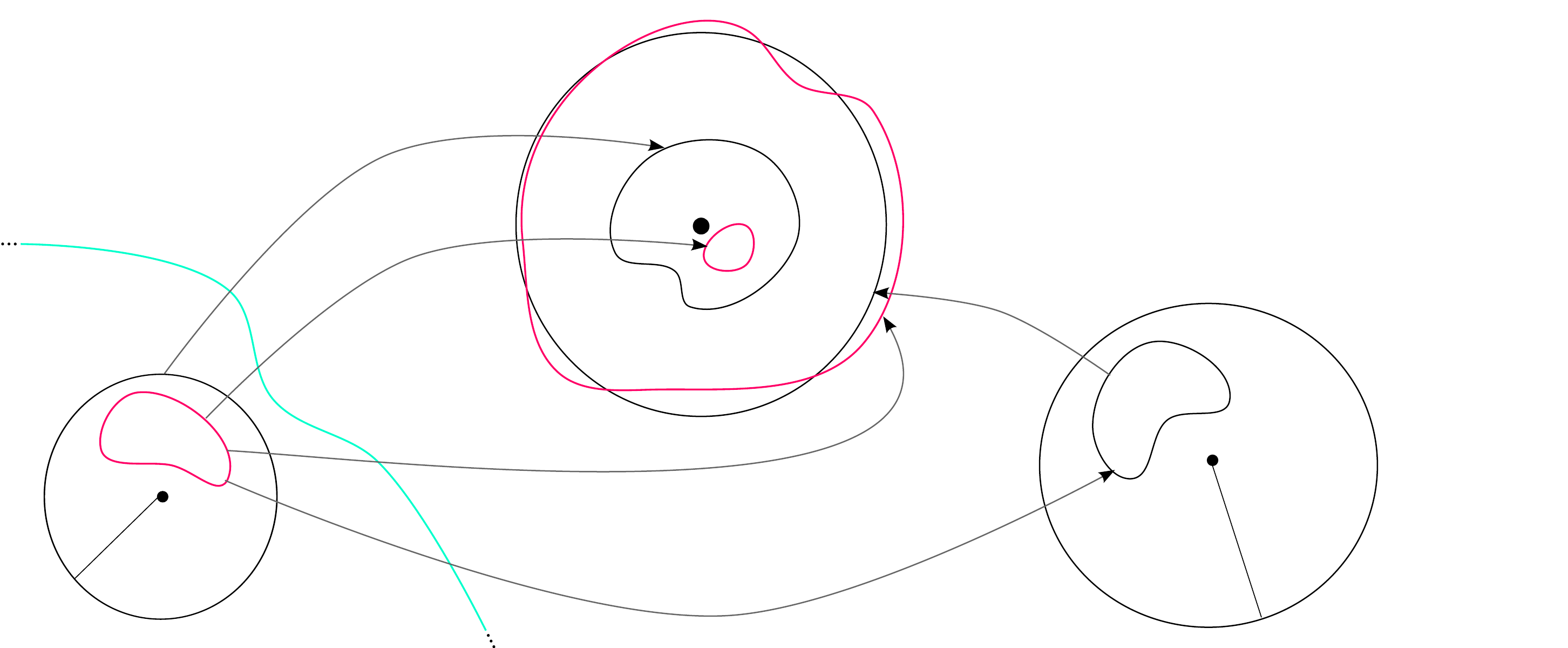}}%
    \put(0.09409166,0.07){\color[rgb]{0,0,0}\makebox(0,0)[lb]{\footnotesize{$\lambda_0$}}}%
    \put(0.4345913,0.28){\color[rgb]{0,0,0}\makebox(0,0)[lb]{\footnotesize{$\gamma(\lambda_0)$}}}%
    \put(0.77932241,0.10462194){\color[rgb]{0,0,0}\makebox(0,0)[lb]{\footnotesize{$\infty$}}}%
    \put(0.075,0.05945605){\color[rgb]{0,0,0}\makebox(0,0)[lb]{\footnotesize{$\delta$}}}%
    \put(0.086,0.136){\color[rgb]{0,0,0}\makebox(0,0)[lb]{\footnotesize{$V$}}}%
    \put(0.54,0.37){\color[rgb]{0,0,0}\makebox(0,0)[lb]{\footnotesize{$D=f_{\lambda_0}(U)$}}}%
    \put(0,0.27){\color[rgb]{0,0,0}\makebox(0,0)[lb]{\small{$M$}}}%
    \put(0.17906989,0.3){\color[rgb]{0,0,0}\makebox(0,0)[lb]{\footnotesize{$\gamma(\lambda)$}}}%
     \put(0.2156329,0.24){\color[rgb]{0,0,0}\makebox(0,0)[lb]{\footnotesize{$\gamma(\lambda)$}}}%
    \put(0.79860981,0.06604038){\color[rgb]{0,0,0}\makebox(0,0)[lb]{\footnotesize{$\epsilon$}}}%
    \put(0.72142123,0.16658866){\color[rgb]{0,0,0}\makebox(0,0)[lb]{\footnotesize{$U$}}}%
    \put(0.60563837,0.23158955){\color[rgb]{0,0,0}\makebox(0,0)[lb]{\footnotesize{$f_{\lambda_0}$}}}%
    \put(0.3360011,-0.01){\color[rgb]{0,0,0}\makebox(0,0)[lb]{\footnotesize{$G(\lambda)=\psi_\lambda^{-1}(f^n_\lambda(v_\lambda))$}}}%
     \put(0.25,0.08){\color[rgb]{0,0,0}\makebox(0,0)[lb]{\footnotesize{$F(\lambda)=f_\lambda^{n+1}(v_\lambda)\sim f_\lambda\circ G(\lambda)$}}}%
  \end{picture}%
\endgroup%
\end{center}
\caption{
\label{fig:pseudomontel} 
\small  An illustration of the proof of Proposition~\ref{pseudomontel}. The final claim follows by Lemma~\ref{shooting}, using  the fact that $\gamma(V)\subset F(V)$ and $\gamma(\partial V)\cap F(\partial V)=\emptyset$. 
}
\end{figure}
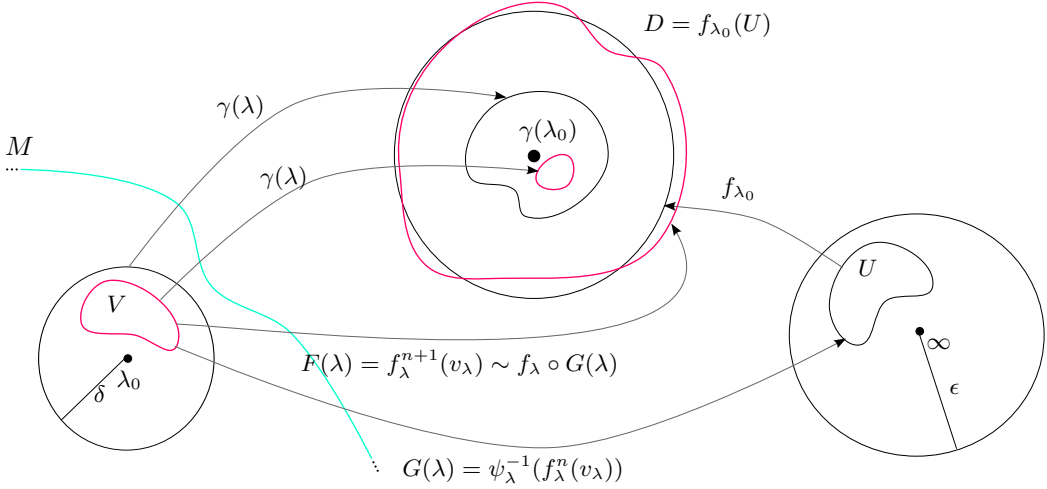

%
%
\end{proof}


\section{Cycles exiting the domain. Proof of Theorem A.} \label{sec:exiting}

The goal in this section is to explore the chain of of implications between cycles disappearing to infinity, the existence of virtual cycles   (Lemma \ref{lem:limcycle}), and  the activity  of at least one asymptotic value or critical point (Subsection \ref{subs:A}), hence proving Theorem A.  

In the process, we shall see that the only possible parameters for which some periodic orbits cannot be followed holomorphically are either parabolic parameters, or those for which  a cycle disappears to infinity; or accumulations thereof (Proposition \ref{prop:onlyobstructions}).

\subsection{Cycles exiting the domain require asymptotic values}
\ 

In this section we show that cycles exiting the domain must create a virtual cycle for the limiting parameter (see Lemma~\ref{lem:limcycle}).

It will be useful first to interpret the concept that  $x_i(\lambda)\ra\infty$ for $\lambda\ra\lambda_0\in M$ in the following more abstract way (c.f.\cite{mss,el} and see Figure~\ref{fig:projection}).

\begin{obs}[The projection map and cycles exiting the domain]\label{defn:pi1}
	For $n \in \N^*$, let 
	\[
	P_n:=\{(\lam,z) \in M \times \C: f_\lam^n(z)=z \},
	\]
	and 
	\[
	\pi_1 : P_n \to M
	\]
	be the projection onto the first coordinate. Then, a cycle of period $n$ exits the domain at $\lam_0$ if and only if $\lam_0$ is an asymptotic value of $\pi_1 : P_n \to M$.
\end{obs}

In other words, a cycle of period $n$ exits the domain at $\lam_0$ if and only if there exists a continuous curve $t \mapsto (\lam(t), z(t))$ in $P_n$ such that $\lim_{t \to +\infty} \lam(t)=\lam_0$ and $\lim_{t \to +\infty} z(t)=\infty$.

%
%

 The set $P_n$ is an analytic hypersurface of $M \times \C$, and by the Implicit Function Theorem it is smooth except possibly at points $(\lam,z)$ 
where $z$ is a periodic point of period dividing $n$ with $(f_\lam^n)'(z)=1$.
Moreover, if $\lam \in M$ is a critical value of $\pi_1 : P_n \to M$, then $f_\lam$ has a parabolic cycle of period dividing $n$ and multiplier $1$.

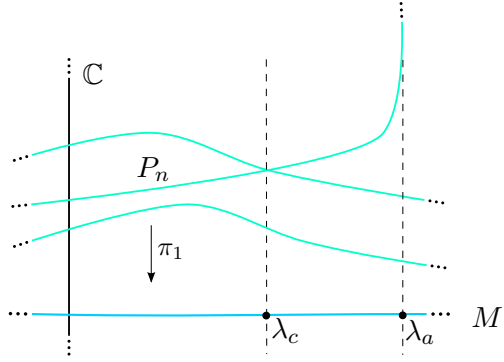
\begin{figure}[hbt!]
\begin{center}
\def\svgwidth{0.4\textwidth}
\begingroup%
  \makeatletter%
    \setlength{\unitlength}{\svgwidth}%
  \makeatother
 \begin{picture}(1,0.89344024)%
    \put(0,0){\includegraphics[width=\unitlength]{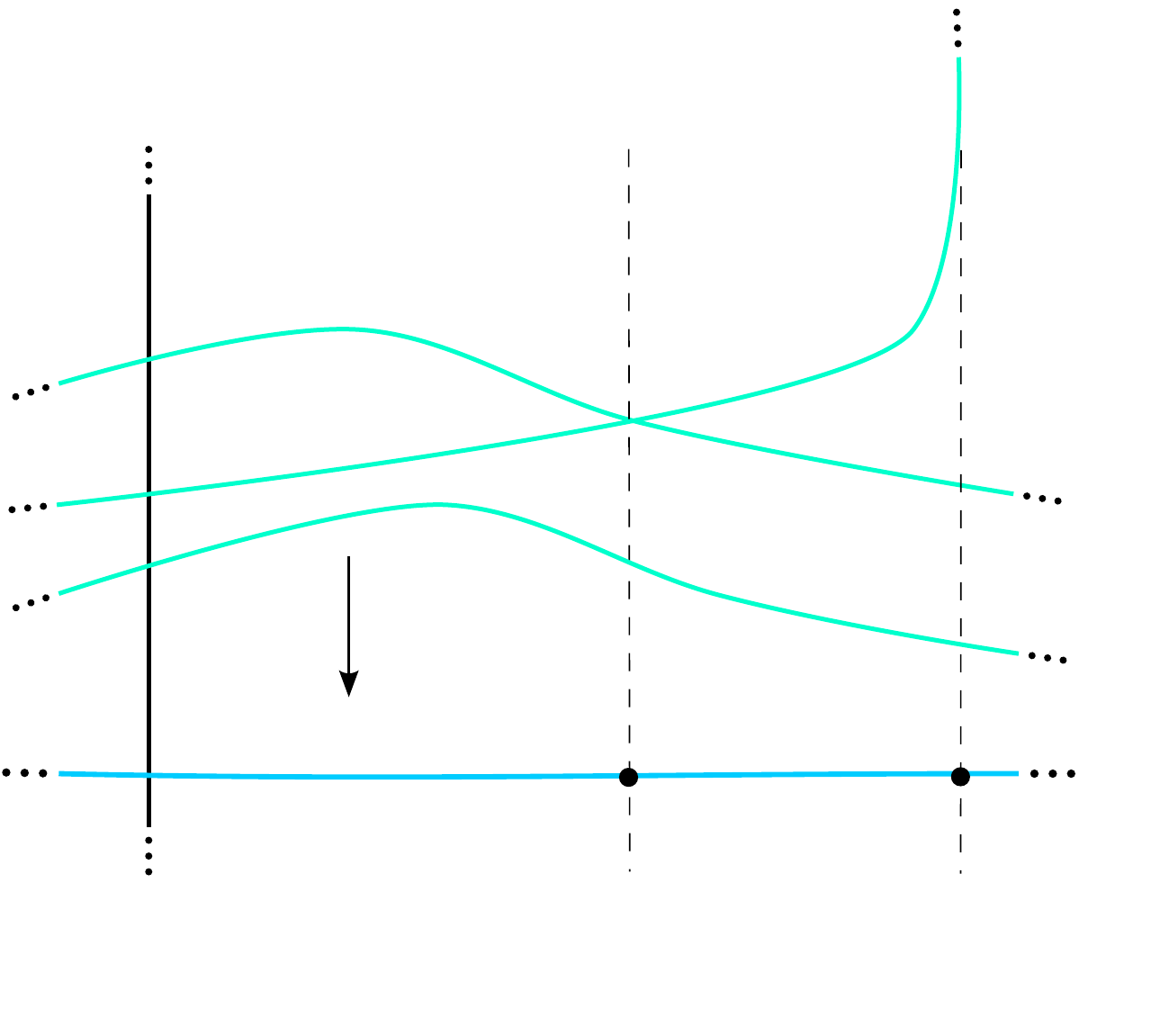}}%
    \put(0.2718644,0.5199087){\color[rgb]{0,0,0}\makebox(0,0)[lt]{\smash{$P_n$}}}%
    \put(0.15773773,0.71703658){\color[rgb]{0,0,0}\makebox(0,0)[lt]{\smash{$\C$}}}%
    \put(0.98,0.2){\color[rgb]{0,0,0}\makebox(0,0)[lt]{\smash{$M$}}}%
    \put(0.31621186,0.35334136){\color[rgb]{0,0,0}\makebox(0,0)[lt]{\smash{\small $\pi_1$}}}%
    \put(0.55691541,0.17615319){\color[rgb]{0,0,0}\makebox(0,0)[lt]{\smash{$\lambda_c$}}}%
    \put(0.83937683,0.17298189){\color[rgb]{0,0,0}\makebox(0,0)[lt]{\smash{$\lambda_a$}}}%
  \end{picture}%
\endgroup%
\end{center}
\caption{\label{fig:projection} \small An illustration of the set $P_n$, the map $\pi_1,$ and its asymptotic and critical values. Here $\lambda_c$ is a critical value for $\pi_1$, corresponding to the map $f_{\lambda_c}$ having a parabolic cycle, and $\lambda_a$ is an asymptotic value for $\pi_1$, corresponding  to  a cycle  exiting the domain at $\lambda_a$.}  

\end{figure}

\begin{defi}[Singular value set of $\pi_1$]
We let $X_n$ denote the singular value set of $\pi_1 : P_n \to M$,  which is the closure of the set of  critical and asymptotic values of $\pi_1$.
Let $X:=\overline{\bigcup_{n=1}^\infty X_n}$.
\end{defi}

The next proposition shows that the singular values of $\pi_1: P_n\ra M$ are the only possible obstructions for a holomorphic  motion of ${\rm Fix}(f_\la^n)$ (the fixed points of $f_\la^n$) to exist, which confirms that $X$ is the appropriate set to study.

\begin{prop}[$\JJ$ moves holomorphically outside $X$]\label{lem:singular values iff no holo motion} \label{prop:onlyobstructions}
	Let $\{f_\la\}_{\la\in M}$ be a natural family of meromorphic maps and $U \subset M$ be a simply connected domain. Let $X, X_n$ be as above. Then 
	\begin{enumerate}[\rm(1)]
		\item $U \cap X_n=\emptyset \iff $ the set ${\rm Fix}(f_\la^n)$ moves holomorphically over $U$ for every $n\geq 1$.
		\item $U \cap X=\emptyset \implies$ the Julia set of $f_\lam$ moves holomorphically over $U$.  
	\end{enumerate}
\end{prop}

\begin{proof} 

To see (1), suppose $U\cap X_n=\emptyset$. Let $\lam_0 \in U$, and let $({ \la_0},z_i)_{i \in \N}$ denote the preimages $\pi_1^{-1}(\lam_0)$.
Then for all $i \in \N$, there exists a holomorphic branch $g_i : U \to P_n$ of $\pi_1^{-1}$ 
with $g_i(\lam_0)=(\lam_0,z_i)$.
In this setting $\lam \mapsto \pi_2 \circ g_i(\lam)$ gives the desired holomorphic motion, where $\pi_2$ is the projection onto the second coordinate.

 For the reverse implication, suppose $\la_0\in U$ is a singular value. If $\la_0$ is an asymptotic value, there is a fixed point of $f_\la^n$ which escapes to infinity when $\la$ approaches $\la_0$. Hence any holomorphic motion of the set ${\rm Fix}_n(f_{\la_0})$ over $U$ could not be surjective. Otherwise, if $\la_0$ is the image of a critical point $(\la_0, z_i)$, every $\la$ in a neighborhood of $\la_0$ will have $k>1$ distinct preimages in $P_n$ splitting off from $z_i$, hence these periodic points cannot be followed holomorphically either.

Statement (2) follows from (1), the $\lam$-lemma and the fact that (repelling) periodic points are dense in the Julia set.
\end{proof}

In fact, it will follow from our results in Section \ref{sec:MSS} that the converse to item (2) also holds,  but for this we shall need the full power of Theorems A and B.

The next lemma shows that any cycle disappearing to infinity requires the help of an asymptotic value which, in the limit, is eventually mapped to infinity, creating a virtual cycle. 

\begin{lem}[Cycle exiting the domain implies virtual cycle for $f_{\la_0}$]\label{lem:limcycle}
	Let $\{f_\la\}_{\la\in M}$ be a natural family of meromorphic maps. Let $t \mapsto (\lam(t), z(t))$ be a curve in $P_n$ with $\limt \lam(t)=\lam_0 \in M$ and $\limt z(t)=\infty$.
	Then there exists a cyclically ordered set $\infty =a_0, \ldots, a_{n-1}\in \hat{\C}$  such that:
	\begin{enumerate}[\rm (1)]
		\item for all $0\leq m \leq n-1$, $a_m=\limt f_{\lam(t)}^m(z(t))$;
		\item if $a_m \in \C$, then $a_{m+1}=f_{\lam_0}(a_m)$;
		\item  if $a_m=\infty$, then $a_{m+1}$ is an asymptotic value  of $f_{\lam_0}$  (possibly equal to $\infty$) and $a_{m-1}$ is either $\infty$ or a pole of $f_{\lam_0}.$
	\end{enumerate}
\end{lem}
In other words, the set $a_0, \ldots, a_{n-1}$ is a \emph{virtual cycle} for $f_{\la_0}$. 
Notice that the lemma  implies that, as $t\to\infty$ (and hence $\la(t)\to\la_0$), either the whole cycle corresponding to $z(t)$ tends to infinity (in which case $\infty$ must be an asymptotic value for $f_{\la_0}$), or there exists at least one finite asymptotic value and one pole in the virtual cycle (possibly more, if there is more than one $a_i$ which equals infinity).  Notice that for this lemma, the finite type assumption is not needed.

\begin{proof}
	To simplify notation, let us denote  $x_m(t):=f_{\lam(t)}^m(z(t))$,  and $f=f_{\lambda_0}$.
	 By assumption \linebreak 
	$\limt f_{\lam(t)}^{n-m}(x_m(t))=\limt z(t)= \infty$, so any finite accumulation point 
	of the curve $t \mapsto x_m(t)$ must be a pre-pole of  $f$ 
	 of order at most $n-m$. In particular, the set of finite accumulation points of this curve is discrete, and so $\limt x_m(t)$ exists (and is possibly $\infty$). Let $a_m:=\lim_{t\to \infty} x_m(t) \in\chat$. 	Item (2) follows easily.
	
	Next, assume that $a_m=\infty$ for some $0 \leq m \leq n-1$. Since $\{f_\lam\}_{\lam \in M}$ is a natural family, we have
\[
x_{m+1}(t) = f_{\lam(t)}(x_{m}(t)) = \phi_{\lam(t)} \circ f \circ \psi_{\lam(t)}^{-1}(x_m(t)),
\]
	where $f:=f_{\lam_0}$, $\phi_{\lam}, \psi_\lam : \hat{\C} \to \hat{\C}$ are quasiconformal homeomorphisms depending holomorphically on $\lam$, 
	and $\phi_{\lam_0}=\psi_{\lam_0}=\id$. Therefore, we have  
	$$f\circ \psi_{\lam(t)}^{-1}(x_m(t))=\phi_{\lam(t)}^{-1}(x_{m+1}(t)),$$
	and $\limt\psi_{\lam(t)}^{-1}(x_m(t))=a_m=\infty$, 
	whereas $\limt \phi_{\lam(t)}^{-1}(x_{m+1}(t))=a_{m+1}$ since $\phi_{\lam(t)}^{-1}$ tends to the identity. 	
	 Therefore $a_{m+1}$ is indeed an asymptotic value of $f$.
	
	Finally, still under the assumption that $a_m=\infty$, it follows from item (2) that if 
 	$a_{m-1}$ is finite then it is a pole.
\end{proof}




%

 Observe that if the periodic point exiting the domain at $\la_0$ is actually a fixed point, then the virtual cycle is just $\{\infty\}$ and hence $\infty$ is a non-persistent asymptotic value of $f_{\la_0}$ (since it cannot be a critical point). This happens e.g. in families where $f_{\lam_0}$ is entire but $f_\lam$ is not, for values of $\lam$ near $\lam_0$, for instance $f_\la(z)=\frac{e^z}{1+\la e^z}$ and $\la_0=0$.

 The rest of this section is dedicated to showing that either one of the asymptotic values in the virtual cycle must be active, or a critical point is, concluding the proof of Theorem A.

\subsection{Preliminary results}

We first record here several lemmas essentially due to Eremenko and Lyubich, some of them modified for our purposes.

Let $\H$ be the left half plane and $T$ a tract over the asymptotic value $v$ (see Definition \ref{defi:tracts}). 
By uniqueness of the holomorphic universal covering up to biholomorphims, there exists a Riemann map $g:\H\ra T$ such that $f\circ g(z)=re^z+v $ for every $z\in\H$.

The following statement is proven in \cite[Lemma 3]{el} in the case where $f$ is a finite type entire map, but the same proof applies in greater generality, as stated in the following lemma.  For this lemma the asymptotic value under consideration is $v=\infty$.

\begin{lem}[\cite{el}, Lemma 3]\label{el-lemma3}
	Let $R>0$, and let $T \subset \C$ be a simply connected domain whose boundary is a real-analytic simple curve with both endpoints converging to $\infty$. Let $f: T \to \C \backslash \overline{\D(0,R)}$ be a holomorphic universal cover, and
	let $\arg$ denote a branch of the argument on $T$. 
	Let $t \mapsto \gamma(t)$ be a continuous curve such that 
	$\lim \gamma(t)=\infty$ and $\gamma(t) \in T$. Then there exists $t_k \to +\infty$ and a constant $C$ independent of $k$, such that 
	\begin{equation}\label{eq:EL lemma 3}
		\ln^2|f(\gamma(t_k))| + \arg^2f(\gamma(t_k)) \geq C |\gamma(t_k)| \exp \frac{\arg^2 \gamma(t_k)}{\ln |\gamma(t_k)|}.
	\end{equation}
\end{lem}
%
%
%

It will be convenient to introduce the following notation, for curves that are not too far apart from each other in $\log$-coordinates, at any given $t>0$.

\begin{defn}[Equivalence relation $\asymp$]
	Let $\gamma_1,\gamma_2: \R_+\to \C^*$ be two continuous curves, converging either both to $0$ or both to $\infty$. We will write
	$\gamma_1 \asymp \gamma_2$ if there exists a constant $C>1$ such that
	\begin{equation}
		\frac{1}{C}  \ln |\gamma_2(t)| \leq \ln |\gamma_1(t)| \leq C \ln |\gamma_2(t)|
	\end{equation}
	and 
	\begin{equation}
		|\arg \gamma_1(t)-\arg \gamma_2(t)| \leq C \left| \ln|\gamma_1(t)| \right|
	\end{equation}
\end{defn}

Note that this definition makes sense because the arguments $\arg \gamma_i$  are well-defined up to a multiple of $2i\pi$. Also note that $\asymp$ is an equivalence relation.

\begin{rem}\label{rem:power of curves}
	If $\gamma_1, \gamma_2$ are two curves as above and $d \in \Z^*$, 
	then it is easy to see that $\gamma_1 \asymp \gamma_2$ if and only if $\gamma_1^d \asymp \gamma_2^d$,  simply because in log coordinates the map $z\mapsto z^d$ becomes $\omega \mapsto d \omega$.
\end{rem}

The following lemma can be extracted from arguments present in \cite{el}; we include details for the convenience of the reader.

\begin{lem}[$f^{-1}$ preserves $\asymp$]\label{lem:f-1preserves}
	Let $\gamma_1, \gamma_2: \R_+ \to \C^*$ be two curves, and $f$ be a bounded type  meromorphic map. 
	Assume that $\gamma_i(t) \to \infty$ and $f\circ \gamma_i(t) \to \infty$, and that $f\circ\gamma_1 \asymp f\circ\gamma_2$.
	Then $\gamma_1 \asymp \gamma_2$.
\end{lem}

\begin{proof}
	Let $A$ be a punctured disk around $\infty$, and let $G$ denote the union of the tracts $T_i$ such that $f: T_i \to A$ 
	is a universal cover. The set  $G$ is non-empty because under the assumptions of the lemma, $\infty$ is an asymptotic value, and because $f$ has a bounded set of singular values.
	Let $U:=\exp^{-1}(G)$ and 	$$\H_R:=\exp^{-1}(A)=\{z \in \C: \re(z)>R\}$$
	for some $R>0$ depending on the radius of $A$. Then there is a holomorphic map $F: U \to \H_R$ making the following diagram commute:
	$$
	\xymatrix{
		U \ar[r]^F \ar[d]_\exp  & \H_R \ar[d]^\exp \\
		G \ar[r]_f & A
	}$$
	Let $\delta_1, \delta_2$ be two respective lifts of $\gamma_1, \gamma_2$ by $\exp$, chosen to be in the same connected  component $U_0$ of $U$: then $\delta_j = \ln|\gamma_j| + i \arg \gamma_j$, and 
	$F(\delta_i)=\ln |f\circ \gamma_j| + i \arg f \circ \gamma_j$, for $j=1,2$.
	
	Let us denote by $I_t$ the Euclidean segment connecting $F\circ\delta_1(t)$ to $F\circ\delta_2(t)$, by $\leucl$ the Euclidean length, by $m(t)=\min(\re (F\circ\delta_1(t)), \re (F\circ\delta_2(t)))$ and by $M(t)=\max(\re( F\circ\delta_1(t)), \re (F\circ\delta_2(t)))$.
	
	By [\cite{el}, Lemma 1], we have $|F'(z)| \geq \frac{1}{4\pi} (\re\, F(z)- R)$ and $F: U_0 \to \H_R$ is a conformal isomorphism, hence it has a well defined inverse branch  $F_U^{-1}:\H_R\ra U$. Therefore 
\begin{align*}
		|\delta_1(t) - \delta_2(t)| &\leq \leucl(F_U^{-1}(I_t))\leq \sup_{w\in I_t}|(F_U^{-1})'(w)|  \leucl(I_t)\leq\\
		&\leq  \frac{4\pi}{m(t) - R}
		|F\circ \delta_1(t)-F\circ\delta_2(t))| \\
		&\leq  \frac{2\pi}{m(t)- R} \cdot \left(2M(t) + |\arg f\circ \gamma_1(t) - \arg f \circ \gamma_2(t)| \right) \\
		&\leq  \frac{C M(t)}{m(t)- R} = O(1)
\end{align*} 
	where we used $f\circ \gamma_1 \asymp f \circ \gamma_2$  in the last inequality.
	Finally, note that $\delta_1- \delta_2 = \log \gamma_1 - \log \gamma_2 = O(1)$ implies 
	$\gamma_1 \asymp \gamma_2$ (it is in fact much stronger).
\end{proof}

\begin{lem}\label{qcasymp}
	Let $f$ be a meromorphic function of bounded type. 
	Consider a curve $\gamma: \R_+ \to \C^*$ with $\gamma(t)\ra\infty$ as $t\ra +\infty$ and assume  that $f\circ\gamma(t)\ra\infty$ as $t\ra+\infty$.   Let $\{h_t: t \geq 0\}$ be a continuous family of $K$-qc homeomorphisms satisfying the hypothesis of Lemma~\ref{lem:ELLemma4}. Then $h_t \circ \gamma \asymp \gamma$.
\end{lem}

\begin{proof}
	The proof follows directly from Lemma \ref{lem:ELLemma4}.
\end{proof}

We observe here a technical point which plays an important role in the proof of Theorem A:
In Lemma \ref{qcasymp}, it is crucial that $h_t(\infty)=\infty$ for all $t\geq 0$, instead of merely 
having $\lim_{t \to +\infty} h_t(\infty)=\infty$.

The lemma below is a slightly weaker version of Lemma 5 from \cite{el}, that will be sufficient for our purposes.
We include the proof for the convenience of the reader, since it is very short using Lemmas  \ref{lem:f-1preserves} and 
\ref{qcasymp}.

\begin{lem}[Compare \cite{el}, Lemma 5]\label{lem:ELLemma5}
	Let $f$ be a meromorphic function with bounded set of singular values. 
	Consider a curve $\gamma: \R_+ \to \C^*$ with $\gamma(t)\ra\infty$ as $t\ra +\infty$ and assume  that $f\circ\gamma(t)\ra\infty$ as $t\ra+\infty$.   Let $\{h_t: t \geq 0\}$ be a continuous family of $K$-qc homeomorphisms satisfying the hypothesis of Lemma~\ref{lem:ELLemma4}. 
	Then there exists a curve $\tilde \gamma\asymp\gamma$, such that 		
	\begin{equation}\label{eqtn:preimage gammat}
		f\circ\tilde\gamma(t)=h_t\circ f\circ\gamma(t).
	\end{equation}
\end{lem}

\begin{proof}
	 Since $f\circ\gamma(t) \to \infty$, we know that $\infty$ is an asymptotic value. Hence the existence of a curve $\tilde \gamma (t) \to \infty$ satisfying $f \circ \tilde \gamma= h_t \circ f \circ \gamma$
	follows from the observation that $h_t \circ f \circ \gamma(t) \to \infty$, and that $f$ is a covering over a punctured 
	neighborhood of $\infty$.
	
	Then, by Lemma \ref{qcasymp} we have $h_t \circ f\circ\gamma \asymp f\circ \gamma$, so by definition of $\tilde{\gamma}$, we have $f \circ \tilde \gamma \asymp f \circ \gamma$.
	Finally, by Lemma \ref{lem:f-1preserves} we have $\tilde \gamma \asymp \gamma$.
\end{proof}

\subsection{Proof of Theorem A} \label{subs:A}
From now on, we assume that our maps are of finite type. 
	
	By assumption, there is a curve $t \mapsto \lam(t)$ in parameter space with $\lam(t) \to \lam_0$, 
	and a cycle of period $n$ exiting the domain along this curve. For simplicity, we will write $f_t, \phi_t, \psi_t$ 
	instead of $f_{\lam(t)}, \phi_{\lam(t)}, \psi_{\lam(t)}$ and $f$ instead of $f_{\lam_0}$ or $f_0$. Analogously, if $v$ is a singular value for $f$ we define $v(t)=\phi_t(v)$ to be the corresponding value for $f_t$.   
	The assumptions that the singular set is finite is necessary in order to ensure that any  asymptotic value $v$ given by Lemma~\ref{lem:limcycle} is isolated, and hence that the maps $g_t:=\frac{1}{f_t(z)-v(t)}$ are  of bounded type, property which is needed in the proof of Lemma~\ref{lem:ELlemmod}.



Let us introduce some further notations. 
We denote by $x_1(t), \ldots, x_n(t)$  the points of the cycle of period $n$ for $f_t$ which exits the domain, 
and assume without loss of generality that $x_n(t) \to \infty$ as $t\to \infty$ (i.e. $\lambda(t)\to\lambda_0$). 
Recall that by Lemma~\ref{lem:limcycle} the  points $a_i=\limt x_i(t)$, $i=1,\ldots, n$ (with indices taken modulo $n$),  form a virtual cycle, hence at least one of them is an asymptotic value.

Therefore, in order to prove Theorem A, we must prove that if the virtual cycle does not contain any active critical point, then at least one  asymptotic value in that virtual cycle is active.
We assume for a contradiction that all asymptotic relations associated to the limit virtual cycle are preserved (that is, every singular value obtained as a limit of one of the $x_i(t)$ remains in the backward orbit of $\infty$ for $\lambda$ in a neighborhood of $\lambda_0$, and therefore  throughout $M$), and the same for critical points belonging to the virtual cycle.
 
In other words, we assume that for \emph{all} $1 \leq i \leq n$ such that $a_i \in S(f)$, we have $f_t^{n-i}(\phi_t(a_i))=\infty$ for all $t>0$. (Recall that if $a_i$ is a singular value of $f$, then $\phi_t(a_i)$ is a singular value of the same nature for $f_t$).

We define a new family of curves $y_1(t), \ldots, y_n(t)$, which  record the orbits of all asymptotic values involved in the limit cycle (see Figure~\ref{fig:proofEL}). 
More precisely, define
\begin{itemize}
	\item if $x_i(t) \to \infty$, then $y_i(t):=\infty$
	\item if $x_{i-1}(t) \to \infty$, then $x_i(t) \to v_i$, where $v_i$ is some asymptotic value of $f$; then we set 
	$y_i(t):=\phi_t(v_i) \to v_i$, which is an asymptotic value for $f_t$. 
	\item if $y_{i-1}(t) \in \C$, then $y_i(t):=f_t(y_{i-1}(t))$.
\end{itemize}

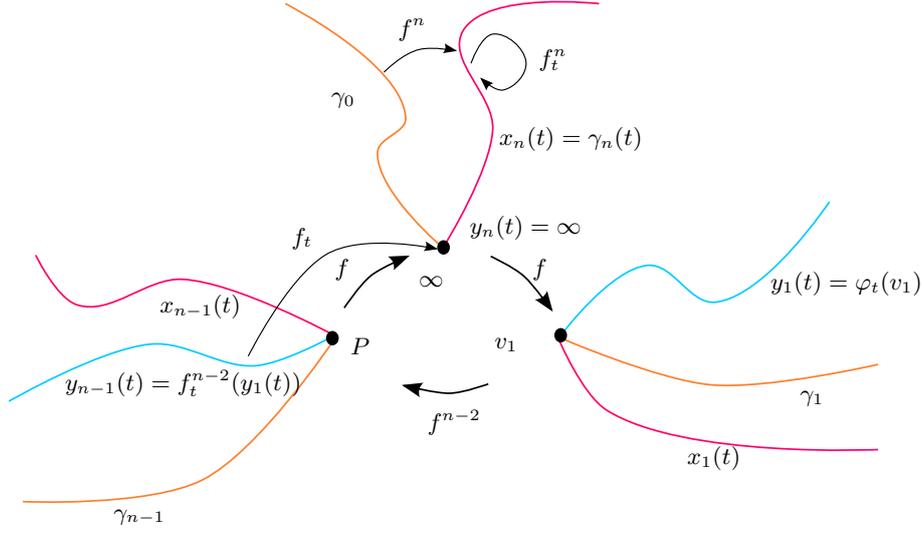
\begin{figure}[hbt!]
	\begin{center}
		\def\svgwidth{0.8\textwidth}
		\begingroup%
		\makeatletter%
		\setlength{\unitlength}{\svgwidth}%
		\makeatother%
		\begin{picture}(1,0.56238517)%
			\put(0,0){\includegraphics[width=\unitlength]{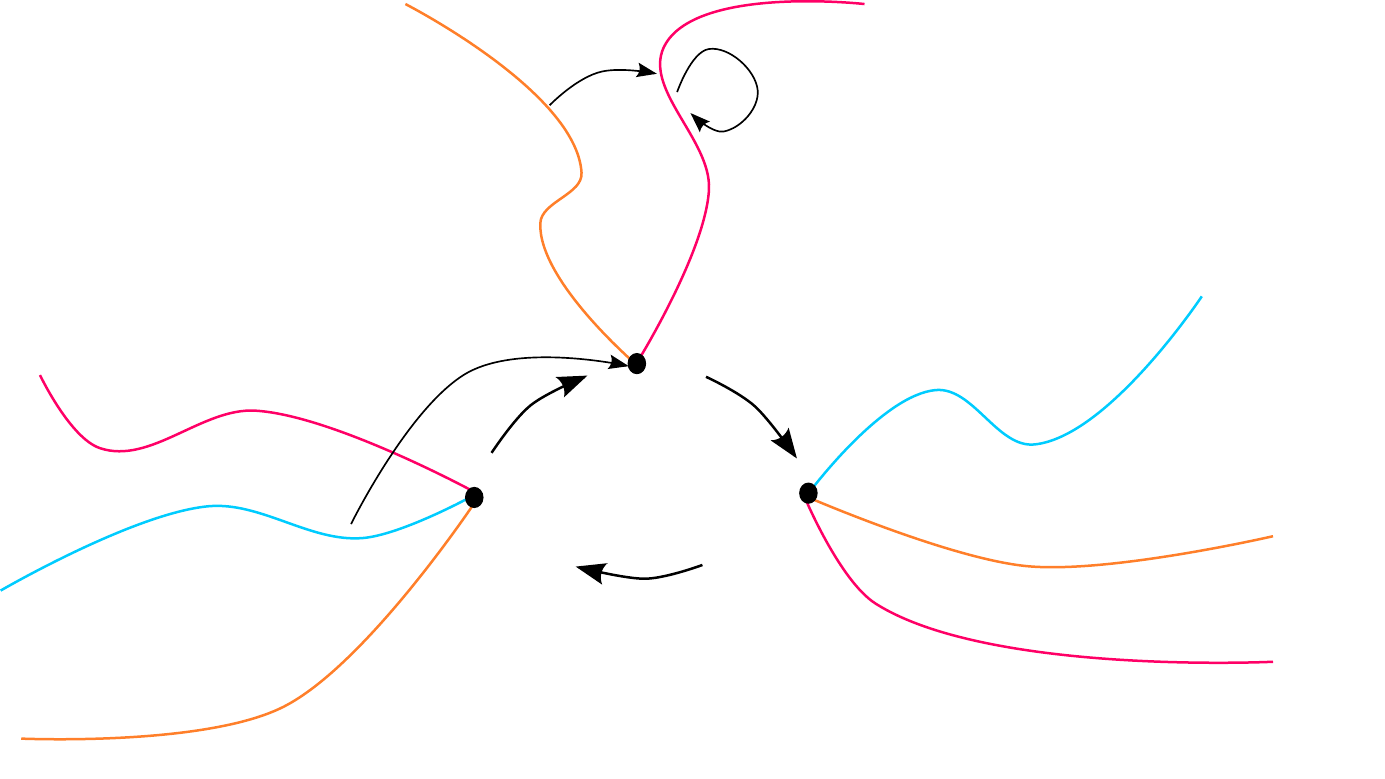}}%
			\put(0.43384346,0.25847996){\color[rgb]{0,0,0}\makebox(0,0)[lb]{\footnotesize{$\infty$}}}%
			\put(0.36161772,0.18625422){\color[rgb]{0,0,0}\makebox(0,0)[lb]{\footnotesize{$P$}}}%
			\put(0.51408657,0.18930485){\color[rgb]{0,0,0}\makebox(0,0)[lb]{\footnotesize{$v_1$}}}%
			\put(0.44303083,0.1){\color[rgb]{0,0,0}\makebox(0,0)[lb]{\footnotesize{$f^{n-2}$}}}%
			\put(0.34490625,0.26442089){\color[rgb]{0,0,0}\makebox(0,0)[lb]{\footnotesize{$f$}}}%
			\put(0.55421969,0.26449877){\color[rgb]{0,0,0}\makebox(0,0)[lb]{\footnotesize{$f$}}}%
			\put(0.51882362,0.40247209){\color[rgb]{0,0,0}\makebox(0,0)[lb]{\footnotesize{$x_n(t)=\gamma_n(t)$}}}%
			\put(0.71777968,0.06275801){\color[rgb]{0,0,0}\makebox(0,0)[lb]{\footnotesize{$x_1(t)$}}}%
			\put(0.16016135,0.22381765){\color[rgb]{0,0,0}\makebox(0,0)[lb]{\footnotesize{$x_{n-1}(t)$}}}%
			\put(0.48769444,0.30773106){\color[rgb]{0,0,0}\makebox(0,0)[lb]{\footnotesize{$y_n(t)=\infty$}}}%
			\put(0.80575344,0.24817958){\color[rgb]{0,0,0}\makebox(0,0)[lb]{\footnotesize{$y_1(t)=\phi_t(v_1)$}}}%
			\put(0.06000657,0.14125758){\color[rgb]{0,0,0}\makebox(0,0)[lb]{\footnotesize{$y_{n-1}(t)=f_t^{n-2}(y_1(t))$}}}%
			\put(0.34152264,0.44848911){\color[rgb]{0,0,0}\makebox(0,0)[lb]{\footnotesize{$\gamma_0$}}}%
			\put(0.11143739,0.00455996){\color[rgb]{0,0,0}\makebox(0,0)[lb]{\footnotesize{$\gamma_{n-1}$}}}%
			\put(0.83688264,0.13043011){\color[rgb]{0,0,0}\makebox(0,0)[lb]{\footnotesize{$\gamma_1$}}}%
			\put(0.29956594,0.29825702){\color[rgb]{0,0,0}\makebox(0,0)[lb]{\footnotesize{$f_t$}}}%
			\put(0.41190165,0.52022156){\color[rgb]{0,0,0}\makebox(0,0)[lb]{\footnotesize{$f^n$}}}%
			\put(0.56078033,0.48773893){\color[rgb]{0,0,0}\makebox(0,0)[lb]{\footnotesize{$f_t^n$}}}%
		\end{picture}%
		\endgroup%
	\end{center}
	\caption{\label{fig:proofEL} \small An illustration of the proof of Theorem A
		in a simple case in which there is only one pole $P$ and one asymptotic value $v_1$ involved. Here $a_n=\infty$, $a_1=v_1$, and $a_{n-1}=P$. Under the contradiction 
		assumption that the singular relation involving $v_1$ is persistent we have that 
		$f_t(y_{n-1}(t))=f_t^{n-1}(\phi_t(v_1))=\infty$. This allows to construct the curves 
		$\gamma_i$ as pullbacks of  the curve $\gamma_n$, obtaining $\gamma_0$ such that 
		$f^n(\gamma_0)=\gamma_n:=x_n$ yet $\gamma_0\asymp \gamma_n$.}
\end{figure}

The assumption that all singular relations 
associated to the limit virtual cycle are preserved implies that this definition is coherent, and that  $y_1(t), \ldots, y_n(t)$ also forms a virtual cycle under $f_t$ for every $t$. In particular, if $x_{i-1}(t)\to \infty$ and $x_i(t)\to a_i=\infty$, this means that $\infty$ is an asymptotic value of $f$ , and the assumption forces it to be persistent, i.e.  $\phi_t(\infty)=\infty=y_i(t)$.
%
Note that we also have $\lim_{t \to +\infty} y_i(t)=a_i$.

The idea of the proof is as follows. Each point on the curve $x_n(t)$ is mapped to itself by $f_t^n$. But since $f_t$ is asymptotically close to $f$ when $t\to\infty$, one could think that these points are mapped ``very close to themselves'' under $f^n$ when $t$ is large enough, in contradiction with Lemma \ref{el-lemma3}. To formalize this idea we consider a third set of curves $\gamma_n(t):=x_n(t), \gamma_{n-1}(t)=f^{-1}(x_n(t)), \ldots, \gamma_0(t)=f^{-n}(x_n(t))$,  the pull backs of  $x_n(t)$ under $f$ (for appropriate branches of the inverse) close to the virtual cycle. We shall show that the $n-$th pullback $\gamma_0$ is very close to $\gamma_n=x_n$ (more precisely  $\gamma_0\asymp\gamma_n$)  while $f^n(\gamma_0)=\gamma_n=x_n$,  which will give  a contradiction through Lemma \ref{el-lemma3}.


We summarize our goal in the following Lemma.

\begin{lem}[Key lemma] \label{lem:gkey} 
	There exist two curves $\gamma_0, \gamma_{n}$ with 
	$$f^{n}(\gamma_{0}(t)) = \gamma_n(t), \quad \lim_{t \to +\infty} \gamma_{0}(t)=\lim_{t \to +\infty} \gamma_n(t)=\infty, \quad \text{ and }\quad \gamma_0 \asymp \gamma_n.$$
	
\end{lem}

\begin{proof}[Proof of Theorem A assuming Lemma \ref{lem:gkey}]
	The map $f^n$ is of finite type and since $\gamma_0$ and its image $\gamma_n$ both tend to $\infty$, it follows that $\infty$ is an  asymptotic value of $f^n$. Hence there exists a simply connected tract $T$ which maps to a punctured neighborhood of $\infty$ as a universal covering, and contains the curve $\gamma_0(t)$ for $t$ large enough.
	
	By Lemma \ref{el-lemma3} applied to $f^n$ and $\gamma_{0}(t)$, we have for all $t$ large enough:
	\begin{equation}
		\ln^2| \gamma_{n}(t_k)| + \arg^2 \gamma_{n}(t_k) \geq 
		 C\exp\left(\ln |\gamma_0(t_k)|\left(1 + \frac{\arg^2\gamma_0(t_k)}{\ln^2 |\gamma_0(t_k)|}\right)\right)
	\end{equation}
	for some sequence $t_k\to \infty$.
	
	On the other hand, by the assumption that $\gamma_{n} \asymp \gamma_{0}$, we have:
	\begin{align}
		\ln |\gamma_{n}(t)| &= \ln|\gamma_{0}(t)|+O(1)\\
		\arg \gamma_n(t) &= \arg \gamma_{0}(t) + O(\ln |\gamma_0(t)|)
	\end{align}
	which leads to a contradiction.
\end{proof}

The proof of  Lemma \ref{lem:gkey} is done by induction. We start with the curve $\gamma_n:=x_n\ra a_n=\infty$. Then, given a curve $\gamma_i(t) \to a_i$ with $i=n\ldots 1$ we will find a curve $\gamma_{i-1}(t)\to a_{i-1}$ which is an appropriate  pullback of $\gamma_i$ under $f$.  This step is  divided into two main cases:  the case in which $a_{i-1} \in \C$ (Lemma \ref{lem:univpullback}) and the case in which $a_{i-1}= \infty$ (Lemma \ref{lem:tractpb}).

\begin{lem}\label{lem:univpullback}
	Let $\gamma_i$ be a curve  such that $\gamma_i(t) \to a_i$ with $\gamma_i(t) \neq a_i$ for all $t>0$, and assume  that $a_{i-1} \in \C$
	and that either $\gamma_i(t)-a_i \asymp x_i(t)-y_i(t)$ (if $a_i \in \C$) or $\gamma_i(t) \asymp x_i(t)$
	(if $a_i=\infty$).  
	Then there exists a curve $\gamma_{i-1}$ such that
	\begin{enumerate}
		\item $f\circ\gamma_{i-1}(t)=\gamma_i(t)$ and $\gamma_{i-1}(t) \neq a_{i-1}$ for all $t>0$
		\item $\gamma_{i-1}(t) \to a_{i-1}$
		\item $\gamma_{i-1}(t)-a_{i-1} \asymp x_{i-1}(t)-y_{i-1}(t)$.
	\end{enumerate}
\end{lem}

\begin{proof}
	First, we choose $\gamma_{i-1}$ to be a lift of $\gamma_i$ by $f$, such that $\gamma_{i-1}(t) \to a_{i-1}$. 
	Note that if $d_i:=\deg(f,a_{i-1})>1$, then there are exactly $d_i$ possible choices (since $\gamma_i(t) \neq a_i$ by assumption). This gives (1) and (2).
	
	Next, we claim that $(\gamma_{i-1}(t)-a_{i-1})^{d_i} \asymp \gamma_i(t)-a_i$ if $a_{\g i} \in \C$,
	and that $(\gamma_{i-1}(t)-a_{i-1})^{-d_i} \asymp \gamma_i(t)$ if $a_i=\infty$. 
	
	This can be seen  easily from the series expansions
	\begin{align*}
		f(z)-a_{i}& =c \cdot (z-a_{i-1})^{d_i} + o((z-a_{i-1})^{d_i}), \text{\ if $a_i\in\C$, or }\\
		f(z) &= c \cdot (z-a_{i-1})^{-d_i} + o((z-a_{i-1})^{-d_i}), \text{\ if $a_i=\infty$},
	\end{align*}
	with $c\neq 0$ (compare Remark~\ref{rem:power of curves}).

	Since critical relations are  assumed to be persistent along the virtual cycle $a_1, \ldots, a_n$, we have $\deg(f_t, y_{i-1}(t))=d_i=\deg(f,a_{i-1})$. Therefore we also have series expansions of the form
	\begin{align*}
		f_t(z)-y_i(t)& =c(t) \cdot (z-y_{i-1}(t))^{d_i} + o((z-y_{i-1}(t))^{d_i}), \text{\ if $a_i\in\C$, or }\\
		f_t(z) &= c(t) \cdot (z-y_{i-1}(t))^{-d_i} + o((z-y_{i-1}(t))^{-d_i}), \text{\ if $a_i=\infty$},
	\end{align*}
	where $c(t) \to c \neq 0$. Since $x_{i+1}(t)=f_t(x_i(t))$, it follows that $(x_{i-1}(t)-y_{i-1}(t))^{d_i} \asymp x_i(t)-y_i(t)$ if $a_{i} \in \C$,
	and $(x_{i-1}(t) - y_{i-1}(t))^{d_i} \asymp x_i(t)$ if $a_i=\infty$.
	
	Therefore:
	\begin{enumerate}
		\item If $a_i=\infty$, then by assumption we have $\gamma_i \asymp x_i$, and we have proved 
		that $(\gamma_{i-1}-a_{i-1})^{d_i}\asymp \gamma_i$ and $(x_{i-1}-a_{i-1})^{d_i} \asymp x_i$;
		therefore $(\gamma_{i-1}-a_{i-1})^{d_i} \asymp (x_{i-1}-a_{i-1})^{d_i}$, 
		which in turn implies $\gamma_{i-1}-a_{i-1} \asymp x_{i-1}-a_{i-1}$ (see again Remark~\ref{rem:power of curves}).
		\item If $a_i \in \C$, then similarly: by assumption, we have $\gamma_i - a_i \asymp x_i - y_i$,
		and we have proved that $(\gamma_{i-1}-a_{i-1})^{d_i} \asymp \gamma_i - a_i$ and $(x_{i-1}-y_{i-1})^{d_i} \asymp x_i-y_i$.
		Therefore we again have $(\gamma_{i-1}-a_{i-1})^{d_i} \asymp (x_{i-1}-a_{i-1})^{d_i}$
		and finally $\gamma_{i-1}-a_{i-1} \asymp x_{i-1}-a_{i-1}$.
	\end{enumerate}
\end{proof}

We now turn to the other case, $a_{i-1}=\infty$.
Before proving the analogue of Lemma \ref{lem:univpullback}, namely Lemma \ref{lem:tractpb},
we will require the following modification of Lemma \ref{lem:ELLemma5} adapted to the case of a finite asymptotic value:

\begin{lem}\label{lem:ELlemmod}
	Let $\gamma(t) \to \infty$ be a curve such that $f_t(\gamma(t)) \to v \in \C$. Then, there exists a curve $\gamma'(t) \to \infty$ such that $\gamma' {\asymp} \gamma$ and $f(\gamma'(t))-v=f_t(\gamma(t))-v(t)$, where $v(t)=\phi_t(v)$.
\end{lem}

\begin{proof}
	Let $g_t(z):=\frac{1}{f_t(z)-v(t)}$, and $g(z):=\frac{1}{f(z)-v}$. Let $M_t(z):=\frac{1}{z-v(t)}$.
	Then observe that $g=M_0 \circ f$, and
	\begin{equation}
		g_t = M_t \circ f_t = \left( M_t \circ \phi_t \circ M_0^{-1} \right) \circ g \circ \psi_t^{-1}
	\end{equation}
	This shows that $g_t$ is a natural family of bounded type meromorphic maps of the form $g_t=\tilde{\phi_t} \circ g \circ \psi_t^{-1}$, with $\tilde{\phi_t}:= M_t \circ \phi_t \circ M_0^{-1}$. Moreover, $\tilde{\phi_t}$ is a quasiconformal homeomorphism of $\hat{\C}$, and $\tilde{\phi_t}(\infty)=\infty$ (since $M_0^{-1}(\infty)=v$ and $M_t \circ \phi_t(v)=\infty$).
	
	Since we have $g_t(\gamma(t)) \to \infty$, we may apply Lemma \ref{lem:ELLemma5} to $g_t$,
	which gives a curve $\gamma'(t) \to \infty$ such that $\gamma' \sim \gamma$, and $g_t(\gamma(t))=g(\gamma'(t))$.
	
	It remains to check that $f(\gamma'(t))-v=f_t(\gamma(t))-v(t)$. But 
	\begin{align*}
		g_t(\gamma(t)) &=g(\gamma'(t)) \\
		\frac{1}{f_t(\gamma(t))-v(t)} &= \frac{1}{f(\gamma'(t))-v}\\
		f(\gamma'(t))-v&=f_t(\gamma(t))-v(t)
	\end{align*}
	and the lemma is proved.
\end{proof}

\begin{lem}\label{lem:tractpb}
	Let $\gamma_i \to a_i$  be a curve such that either $a_i \in \C$ and
	$\gamma_i(t) - a_i  \asymp x_i(t) - y_i(t)$, or $a_i=\infty$ and $\gamma_i(t) \asymp x_i(t)$. Assume further that $a_{i-1}=\infty$. Then there exists a curve $\gamma_{i-1}$ such that 
	\begin{enumerate}
		\item $f\circ \gamma_{i-1}(t)=\gamma_i(t)$ and   $\gamma_{i-1}(t) \neq \infty$ for all $t >0$
		\item $\gamma_{i-1}(t) \to \infty$ 
		\item $\gamma_{i-1}(t) \asymp x_{i-1}(t).$
	\end{enumerate}
\end{lem}

\begin{proof}
	We will distinguish two cases: $a_i=\infty$ or $a_{i} \in \C$.
	
	First, assume that $a_i=\infty$. In that case, $\infty$ is an asymptotic value for $f$ and by assumption it remains an asymptotic value for $f_t$, so that $\phi_t(\infty)=\infty$. Moreover, note that $x_i(t)=f_t(x_{i-1}(t)) = \phi_t \circ f \circ \psi_t^{-1} \circ x_{i-1}(t)$, 
	and that $\psi_t^{-1}(x_{i-1}(t))$ is a curve that tends to $\infty$.
	Therefore we can apply Lemma \ref{lem:ELLemma5} with $h_t:=\phi_t$ (since, again, $\phi_t(\infty)=\infty$) 
	and $\gamma(t):=\psi_t^{-1}(x_{i-1}(t))$. We obtain in this way a curve $\tilde \gamma$ such that 
	$\tilde \gamma(t) \to \infty$,  $\phi_t \circ f \circ \psi_t^{-1}(x_{i-1}(t))=x_{i}(t) = f(\tilde \gamma(t))$,
	and $\tilde \gamma(t) \asymp \psi_t^{-1} \circ x_{i-1}(t)$. By Lemma \ref{qcasymp} we have
	$x_{i-1}(t) \asymp \psi_t^{-1}(x_{i-1}(t))$ (since $\psi_t^{-1}(\infty)=\infty$).
	So $\tilde \gamma(t) \asymp x_{i-1}(t)$.
	
	Moreover, we have $f(\tilde \gamma) = x_i \asymp \gamma_i$ by assumption. Let $\gamma_{i-1}$ be a lift of $\gamma_i$ by $f$:
	then $$f\circ \gamma_{i-1} = \gamma_i \asymp x_i = f \circ \tilde \gamma_i,$$
	so that by Lemma \ref{lem:f-1preserves} we have $\tilde \gamma \asymp \gamma_{i-1}$.
	Finally, we have:
	$$\gamma_{i-1}\asymp \tilde \gamma \asymp x_{i-1},$$
	and we are done in this case.
	
	\medskip

	We now treat the case when $a_{i} \in \C$. In that case, we apply Lemma \ref{lem:ELlemmod} with $\gamma:=x_{i-1}$ and get 
	a curve $\tilde \gamma$ such that $\tilde \gamma \asymp x_{i-1}$ and $f\circ \tilde \gamma - a_i= f_t \circ x_{i-1} - y_i= x_i - y_i$.
	
	Let $\gamma_{i-1}$ be a lift by $f$ of $\gamma_i$, such that $\gamma_{i-1}(t) \to \infty$.
	It remains to argue as above that 
	$\tilde \gamma\asymp \gamma_{i-1}$.  But this follows precisely
	from the same Lemma \ref{lem:f-1preserves} applied to $g:=\frac{1}{f-a_i}$ instead of $f$, since by assumption $\gamma_i - a_i \asymp x_i - y_i$ and therefore
	$$f \circ \tilde \gamma - a_i = x_i - y_i  \asymp f \circ \gamma_{i-1} - a_i.$$
	Then finally we also have 
	\begin{equation}
		x_{i-1} \asymp \tilde \gamma \asymp \gamma_{i-1},
	\end{equation}
	and the lemma is proved.
\end{proof}

We are now finally ready to prove the key Lemma \ref{lem:gkey},
which will conclude the proof of Theorem A.

\begin{proof}[Proof of Lemma \ref{lem:gkey}]
	
	We define $\gamma_n(t):=x_n(t)$, and then proceed by induction to construct curves $\gamma_i$ such that 
	$\gamma_i(t) \to a_i$, $f^{n-i}(\gamma_{i}(t))=\gamma_n(t)$, and:
	\begin{itemize}
		\item if $a_i \neq \infty$, then  $\gamma_i - a_i \asymp x_i - y_i$ 
		\item if $a_i=\infty$, then $\gamma_i \asymp x_i$.
	\end{itemize}

	Assume $\gamma_i$ is constructed. We then have two cases: either $a_{i-1}=\infty$ or not.
	If $a_{i-1} \neq \infty$, then we apply Lemma \ref{lem:univpullback}.
	Otherwise, we apply Lemma \ref{lem:tractpb}. In either case, the induction is proved.
\end{proof}

%
\section{Existence of attracting cycle exiting the domain at virtual cycle parameters. Proof of Theorem B.}\label{sect:Accessibility}

The goal in this section is to prove Theorem B, the Accessibility Theorem. We start with  a  lemma  which was kindly pointed out to us by Lasse Rempe. 

\begin{lem}\label{lem:lasse} Let $T$ be a simply connected hyperbolic  domain, $\rho_T$ be the hyperbolic density in $T$, and let $z,w\in T$. Then 
\begin{equation}\label{eq:lasse}
\dist_T(z,w)\geq\frac{1}{2}\left|\ln\frac{\dist(w,\partial T)}{\dist(z,\partial T)}\right|.
\end{equation}
\end{lem}
\begin{proof} Let $\Gamma$ be the collection of all curves $\gamma:[0,t_\gamma] \ra T$, parametrized in arclength, for which $\gamma(0)=z$ and $\gamma(t_\gamma)=w$. Notice that by the triangular inequality, given such a curve $\gamma$ parametrized by a parameter $t$, for any point $\gamma(t)\in\gamma$ we have 
\begin{align}\label{eq:triangular1}
\dist(\gamma(t), \partial T)&\leq \dist(z, \partial T)+t \quad \text{and in particular}\\\label{eq:triangular2}
\dist(w, \partial T)&\leq \dist(z, \partial T)+t_\gamma.
\end{align}  
For  a curve $\gamma\in\Gamma$ denote by $\ell_T(\gamma)$ its hyperbolic length in $T$. Then 
$\dist_T(z,w)=\inf_{\gamma\in\Gamma}\ell_T(\gamma)$. 
On the other hand, consider any such $\gamma$. Using  standard estimates for the hyperbolic metric (see \cite{BeardonMinda}) as well as (\ref{eq:triangular1})and (\ref{eq:triangular2})) we obtain
\begin{align*}
\ell_T(\gamma)&=\int_0^{t_\gamma}\rho_T(\gamma(t))dt\geq\int_0^{t_\gamma}\frac{1}{2\dist(\gamma(t), \partial T)}  dt\geq\\
& \int_0^{t_\gamma}\frac{1}{2\dist(z, \partial T)+t}  dt=\frac{1}{2}\ln\frac{\dist(z, \partial T)+t_\gamma}{\dist(z, \partial T)}\geq \frac{1}{2}\ln\frac{\dist(w, \partial T)}{\dist(z, \partial T)}. 
\end{align*}
This proves the claim.
\end{proof}
Using Lemma \ref{lem:lasse} and standard hyperbolic estimates, one can show that the Riemann map from the left half plane to a tract, restricted to the negative real axis, cannot contract too much when approaching infinity, no matter what the geometry of the tracts is.  In other words, the central curve inside a tract parametrized by the negative half line cannot converge to infinity too slowly.

\begin{lem}[Asymptotic derivative of the Riemann map]\label{lem:nocusp} Let $\H$ be the left half plane,  $T$ be a simply connected hyperbolic  domain, $g:\H\ra T$ be a Riemann map. Then for every $\alpha>0$, 
\begin{equation}\label{eq:nocusp}
\lim_{t \to +\infty} |g '(-t)| e^{\alpha t} =\infty. 
\end{equation}
\end{lem}
\begin{proof} Let $C=\dist_T(g(-1),\partial T)$. Then (\ref{eq:lasse}) gives 
\begin{equation}\label{eq:lassespecial}
\ln t=\dist_\H(-1,-t)=\dist_T(g(-1),g(-t))\geq\frac{1}{2}\left|\ln\frac{C}{\dist(g(-t),\partial T})\right|.
\end{equation}
By definition of hyperbolic density, 
$$
\rho_T(g(-t))=\frac{1}{|g'(-t)|}\rho_\H(-t)=\frac{1}{t|g'(-t)|}, 
$$
and by the standard estimates on the hyperbolic density on simply connected sets (see e.g. \cite{BeardonMinda}), 
\begin{equation}\label{eq:estimatesgprime}
\rho_T(g(-t))=\frac{1}{t|g'(-t)|} \leq \frac{2}{\dist(g(-t), \partial T)}
\end{equation}
 Suppose first that  for a given $t$, $C\leq \dist(g(-t),\partial T)$; then using (\ref{eq:estimatesgprime}), for every $\beta>0$ and whenever $t$ is sufficiently large (depending only on $C$) we obtain that 
 $$
|g'(-t)|\geq \frac{\dist(g(-t),\partial T)}{2t}\geq  \frac{C}{2t} \geq e^{-\beta t}.
 $$
Otherwise, for any  $t$ such that  $C\geq \dist(g(-t),\partial T)$  we can remove the modulus in the right hand side of  (\ref{eq:lassespecial}) to obtain
 \begin{align*}
 2\ln t&\geq\ln\frac{C}{\dist(g(-t),\partial T)}\\
 t^2&\geq\frac{C}{\dist(g(-t),\partial T)}  \geq \frac{C}{2t|g'(-t)|}.\\
 \end{align*}
 Hence  for every $\beta>0$ and every $t$ sufficiently large depending only on $C$
 $$
 |g'(-t)|\geq \frac{C}{2t^3}\geq e^{-\beta t}.
 $$
\end{proof}
%
%

%


\begin{proof}[Proof of Theorem B]
To simplify the notations, set  $f:=f_{\lambda_0}$. 
	 Let $v=v_{\la_0}$ be an  asymptotic value such that $f^n(v)=\infty$. Recall that $f_\lam=\phi_\lam \circ f \circ \psi_\lam^{-1}$. Let $V:=\D^*(v, r)$ be a punctured disk centered at $v$ disjoint from $S(f)$, and $T$ a tract, so that
	$f: T\to V$ is a universal cover. 
	Let $\Phi: T \to \H$ be a conformal isomorphism, where 
	$\H$ is the left half plane.
	In particular, $f(z)=v+r e^{\Phi(z)}$ for all $z \in T$. See Figure \ref{accessibility1}.
	
	Let $V_\lam:=\phi_\lam(V)$ and $T_\lam:=\psi_\lam(T)$, so that 
	$f_\lam:T_\lam\ra V_\lam$ is also  an infinite degree universal cover,
	 and let $\Phi_\lam:=\Phi\circ \psi_\lam^{-1} : T_\lam \to \H$.
	Then $\phi_\lam^{-1} \circ f_\lam : T_\lam \to V$ is a universal cover, and so
	for all $z \in T_\lam$, 
	\begin{equation}\label{eq:expf}
		f_\lam(z)=\phi_\lam\left(v+ r e^{\Phi_\lam(z)}\right)
	\end{equation}

	Now, we wish to find a curve $t \mapsto \lam(t)$ in parameter space such that 
	\begin{equation} \label{eq2}
		\Phi_{\lam(t)} \circ f_{\lam(t)}^{n}(v_{\lam(t)}) = -t.
	\end{equation}

	We use the same notations as in the proof of Proposition \ref{pseudomontel}: we let $G(\lam):=\psi_\lam^{-1} \circ f_\lam^{n}(v_\lam)$  and recall that $G(\la_0)=\infty$. Given the definition of $\Phi_\lam$, Equation (\ref{eq2}) is equivalent to 
	\begin{equation}\label{eq:curve}
	\Phi\circ \psi_{\lam(t)}^{-1} \circ f_{\lam(t)}^{n} (v_{\lam(t)}) = -t,
	\end{equation}
	or
	\begin{equation}
		G(\lam(t)) = \Phi^{-1}(-t).
	\end{equation}
	Note that $t \mapsto \Phi^{-1}(-t)$ 	is a curve such that $\lim_{t \to +\infty}  \Phi^{-1}(-t)=\infty$. By Remark \ref{rem:Gbranched}, the map $G$ is locally a branched cover over a neighborhood of $\infty$, and so we can find the desired curve $t \mapsto \lam(t)$ (defined for $t$ large enough, and possibly not unique). See Figure \ref{accessibility1}.
	
\begin{figure}[hbt!]
\centering
\setlength{\unitlength}{0.8\textwidth}
\includegraphics[width=0.8\textwidth]{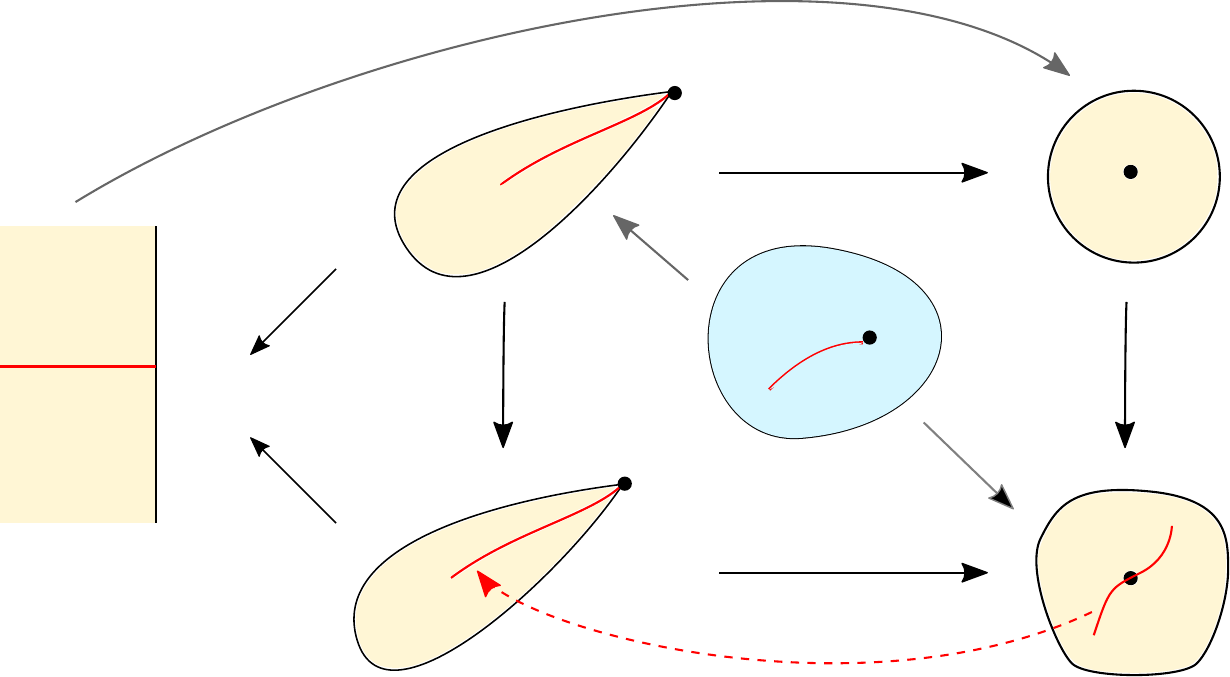}
\put(-0.33,0.42){$f$}
\put(-0.33,0.1){$f_\la$}
\put(-0.33,0.03){$f_\la^n$}
\put(-0.65,0.44){$T$}
\put(-0.67,0.14){$T_\la$}
\put(-0.63,0.24){$\psi_\la$}
\put(-0.078,0.24){$\phi_\la$}
\put(-0.79,0.3){$\Phi$}
\put(-0.79,0.13){$\Phi_\la$}
\put(-0.44,0.47){$\infty$}
\put(-0.48,0.15){$\infty$}
\put(-0.07,0.4){$v$}
\put(-0.07,0.06){$v_\la$}
\put(-0.03,0.46){$V$}
\put(-0.03,0.15){$V_\la$}
\put(-0.28,0.34){$M$}
\put(-0.28,0.26){$\la_0$}
\put(-0.7,0.52){$v+r e^{z}$}
\put(-0.865,0.24){$0$}
\put(-0.95,0.33){$\H$}
\put(-0.46,0.35){$G$}
\put(-0.21,0.175){\small$v(\la)$}
\caption{\label{accessibility1} Setup of the Proof of Theorem B.}
\end{figure}
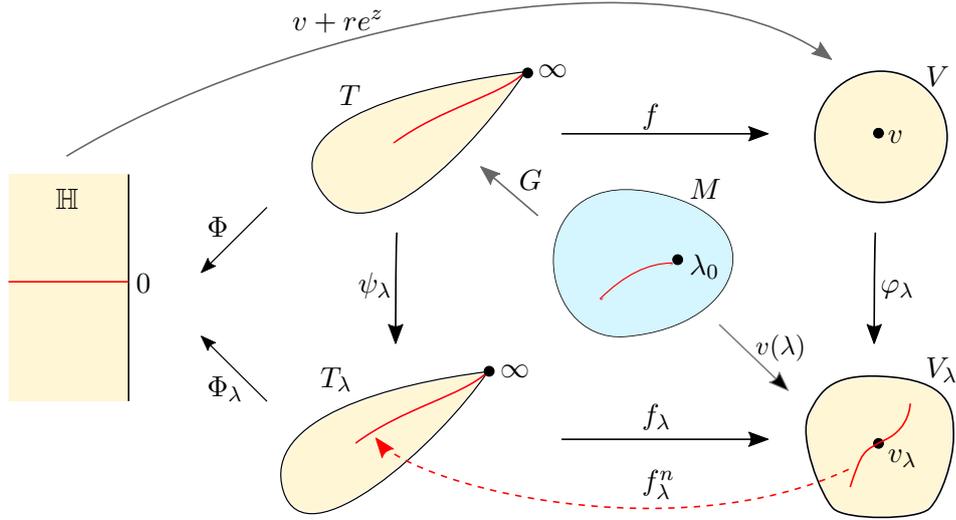
	
	Now let $D_t:=\Phi_{\lam(t)}^{-1}(\D(-t,\pi))\subset T_{\lam(t)}$ and let $U_t$ denote the connected component of 
	$f_{\lam(t)}^{-n}(D_t)$ containing $v_{\lam(t)}$. We will prove that for all $t$ large enough, 	$f_{\lam(t)}(D_t) \Subset U_t$, or equivalently, 
	$f_{\lam(t)}^{n+1}(U_t) \Subset U_t$; this implies the existence of an attracting fixed point for $f_{\lam(t)}^n$. 
	
	First, let us show that for $r$ small and $t$ large, $f_{\lam(t)}(D_t)$ is contained in a small disk centered at $v_{\lam(t)}$, or more precisely,   
	\begin{equation}\label{eq:controldt}
		f_{\lam(t)}(D_t) \subset \D\left(v_{\lam(t)}, e^{-t(1-\epsilon)})\right).
	\end{equation}
		By (\ref{eq:expf}) we have that for all $z \in \H$.
	 \[
	 f_\lam \circ \Phi_\lam^{-1}(z)=\phi_\lam(v +r e^z).
	 \]
Since $\D(-t,\pi)\subset\{z\in\C: \Re z< -t+\pi\}$ we have that 
	\[
	f_{\lam(t)}(D_t) \subset \phi_{\lam(t)}(\D(v, r e^{-t+\pi}))
	\]
	 Let $\epsilon>0$. By Lemma \ref{lem:distortion}, we have for all $t$ large enough:
	\begin{equation}
		f_{\lam(t)}(D_t) \subset \D\left(v_{\lam(t)}, (r e^\pi)^{1-\epsilon}e^{-t(1-\epsilon)})\right),
	\end{equation}
which for $r$ small implies (\ref{eq:controldt}). See Figure \ref{accessibility2}.
	
	Now we show that $U_t$ contains a disk centered at $v_\lam(t)$ whose radius, for $t$  large,  is much larger than $e^{-t(1-\epsilon)}$.
	
	
	Let us first estimate $\dist(f_{\lam(t)}^{n}(v_{\lam(t)}), \partial D_t)$. 
	To lighten the notations, let $g:=\Phi^{-1}$; then $g$ is univalent on $\H$
	and $D_t=\psi_{\lam(t)} \circ g(\D(-t,\pi))$. By Koebe's theorem, 
	$g(\D(-t,\pi))$ contains a disk 
	\[
	\D(g(-t),C |g'(-t)|).
	\]
	Then, by Lemma \ref{lem:distortion} and (\ref{eq:curve})
	\[
	D_t=\psi_{\lam(t)}\circ g(\D(-t,\pi))\supset	\D(\psi_{\lam(t)} \circ g(-t),  C^{1+\epsilon}|g'(-t)|^{1+\epsilon})
	=	\D( f_{\lam(t)}^{n}(v_{\lam(t)}), C^{ 1+\epsilon}|g'(-t)|^{1+\epsilon}).
	\]

	Now note that as $t \to +\infty$, $D_t$ is arbitrarily close to $\infty$. In particular, we may assume that for all $t$ large enough
	$D_t \cap S(f_{\lam(t)}^{n})=\emptyset$; hence  since $D_t$ is simply connected, we can define an inverse branch $h_t: D_t \to U_t$ of $f_{\lam(t)}^{-n}$. In fact, $h_t$ can be extended to some 
	simply connected neighborhood of $\infty$ independent from $t$, and as $t \to +\infty$ it converges on that domain
	to an inverse branch of $f^{-n-1}$; in particular, its spherical derivative $h_t^\#(f_{\lam_t}^{n}(v_{\lam(t)}))$ is bounded independently from $t$.
	
	Again, Koebe's theorem applied to $h_t:  \D\left(f_{\lam(t)}^{n}(v_{\lam(t)}), C^{ 1+\epsilon} |g'(-t)|^{1+\epsilon} \right) \to U_t$
	implies that there exists a constant $C'>0$ such that
	\begin{equation}\label{eq:controlut}
		\D(v_{\lam(t)}, C' |g'(-t)|^{1+\epsilon}) \subset U_t.
	\end{equation}
	
	See Figure \ref{accessibility2}. Finally, from equations \eqref{eq:controldt} and \eqref{eq:controlut}, 
	it is enough to prove that 
	\begin{equation}\label{eq:ratio of radii}
	 \frac{e^{-t(1-\epsilon)}}{C' |g'(-t)|^{1+\epsilon}} \to 0 \quad \text{as $t \to +\infty$}, 
	\end{equation}
	which follows from Lemma~\ref{lem:nocusp}. 
	This proves that $f_{\lam(t)}^{n}(U_t) \Subset U_t$, and the theorem then follows from Schwartz's lemma. Note that (\ref{eq:ratio of radii}) also implies that  the multiplier   goes to zero
	as $t \to +\infty$,  since the modulus of $U_t\setminus \overline{f_\la^{n+1}(U_t)}$ tends to infinity.
\end{proof}

\begin{figure}[hbt!]
\centering
\setlength{\unitlength}{0.8\textwidth}
\includegraphics[width=0.8\textwidth]{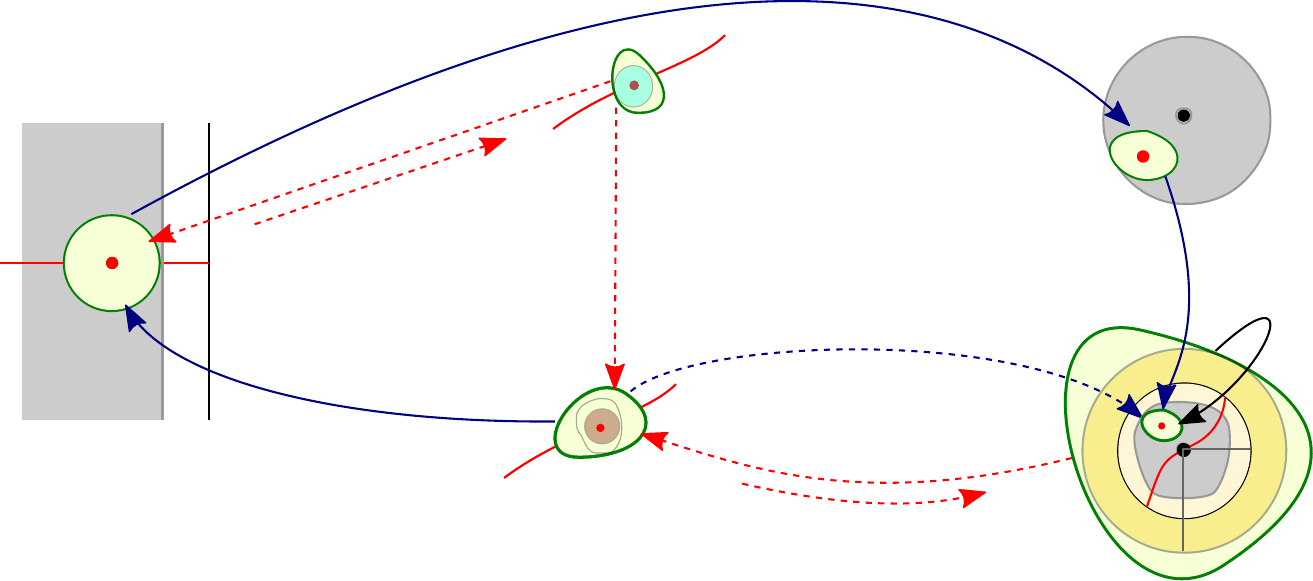}
\put(-0.33,0.19){$f_\la$}
\put(-0.33,0.1){$f_\la^n$}
\put(-0.33,0.03){$h$}
\put(-0.57,0.24){$\psi_\la$}
\put(-0.078,0.24){$\phi_\la$}
\put(-0.75,0.32){$\Phi$}
\put(-0.74,0.27){$g$}
\put(-0.75,0.1){$\Phi_\la$}
\put(-0.09,0.35){$v$}
\put(-0.09,0.085){\scriptsize $v_\la$}
\put(-0.075,0.105){\scriptsize $R_1$}
\put(-0.095,0.04){\scriptsize $R_2$}
\put(-0.03,0.46){$V$}
\put(-0.03,0.15){$U_t$}
\put(-0.59,0.15){$D_t$}
\put(-0.7,0.41){$v+r e^{\Phi}$}
\put(-0.835,0.232){\small $0$}
\put(-0.92,0.36){$\H$}
\caption{\label{accessibility2} Construction of an attracting cycle in the Proof of Theorem B, with $R_1:=e^{-t(1-\epsilon)} \ll C' |g'(-t)|^{1+\epsilon} =:R_2 $.}
\end{figure}


%
%
%

\section{Bifurcation locus} \label{sect:bifloc}

In this section we concern ourselves with the set of  $J$-unstable parameters, the {\em bifurcation locus}. Our goal is to show different characterizations of this set in the spirit of the Ma\~{n}\'e-Sad-Sullivan Theorem, leading to the proof that its complement, the set of $\JJ$-stable parameters,  is open and dense in the parameter space.

To that end, we first show several approximation and density results which will be useful and are interesting on their own.  Since we have not yet formally linked $\JJ$-stability to activity of singular values (Theorem E), we will  first show density of specific types of  parameters in the activity locus, as defined below.

\begin{defi}[Activity locus]Given a holomorphic family $\{f_\lambda\}_{\lambda\in M}$ we define the \emph{activity locus} $\AA$ as the set of parameters in $M$ for which some singular value is active (see Definition~\ref{defn:active singular values}).   The activity locus $\AA(v_{\lambda})$ of a singular value $v_\lambda$ is the set of parameters in $M$ for which the singular value $v_\lambda$ is active. 
\end{defi}

The following Lemma shows that the activity locus of a natural family is a relatively large set in parameter space, a fact which will be useful in the density proofs. 

\begin{lem}\label{lem:perturbA}
	Let $\{f_\lam\}_{\lam \in M}$ be a natural family of finite type meromorphic maps, and let $\mathcal{A}(v_\lam)$
	be the activity locus of a marked singular value $v_\lam$. Then $\mathcal{A}(v_\lam)$ is nowhere locally contained in a proper analytic subset of $M$. More precisely, if $\lam_0 \in \mathcal{A}(v_\lam) \cap H$, where $H \subset M$ is a proper analytic subset, then for every neighborhood $U$ of $\lam_0$ in $M$, $U \cap (\mathcal{A}(v_\lam) \setminus H) \neq \emptyset$.
\end{lem}

\begin{proof}
	Let $\lam_0 \in \mathcal{A}(v_\lam)$, $H$ an analytic hypersurface of $M$ containing $\lam_0$, and $U$ be a small polydisk centered at $\lam_0$ in $M$. Assume for a contradiction that 
	$\mathcal{A}(v_\lam) \cap U \subset H \cap U$. Let $h_n(\lam):=f_\lam^n(v_\lam)$, wherever this expression is well-defined. Let $z_{\lam_0}$ be a repelling periodic point of period at least 3 for $f_{\lam_0}$ such that $z_{\lam_0} \notin S(f_{\lam_0})$, and let $z_\lam$ be the corresponding repelling periodic point for $f_\lam$ given by the Implicit Function Theorem. Up to reducing $U$, we may assume that $\lam \mapsto z_\lam$ is 
	defined over $U$.

	Since $\acal(v_\lam) \neq \emptyset$, there is no $N \in \N$ such that $h_N \equiv \infty$ on $U$;
	it follows from the definition of the activity locus and our assumptions that if $\lam \in U$ and $n \in \N$ are such that $h_n(\lam)=\infty$, then $\lam \in H$.

	We now distinguish two cases: 
	\begin{enumerate}
		\item either there exists $n_0 \in \N$ and $\lam_1 \in U$ such that $h_{n_0}(\lam_1) = \infty$;
		\item or for every $n \in \N$, $h_n$ is well-defined over $U$ and $h_n(U) \subset \C$.
	\end{enumerate}

	Let us first treat case (1). Let $D$ be a holomorphic disk passing through $\lam_1$ and not contained in $H$. Then by the choice of $\lam_1$ and  our previous observation, $h_{n_0}(\lam_1)=\infty$ and $h_{n_0}(\lam) \neq \infty$ for every $\lam \in D \setminus \{\lam_1\}$. By the Shooting Lemma (Proposition \ref{pseudomontel}) applied with $\gamma(\lam):=z_\lam$, we find some $\lam_2 \in D \setminus \{\lam_1\}$ such that 
	$h_{n_0+1}(\lam_2)=z_{\lam_2}$, in other words, $v_{\lam_2}$ is Misiurewicz. Therefore $\lam_2 \in \acal(v_{\lam_2})$, but $\lam_2 \notin H$, a contradiction.

	Case 2 follows from a similar but more classical application of Montel's theorem.
\end{proof}

\subsection{Proof of Theorem D and other density theorems} \label{sec:density}

	For the purpose of this section it is convenient to introduce some additional notation.
	
	\begin{defn}[Truncated parameter]
		We will say that a parameter $\la_0\in M$ is a {\em truncated parameter} if there exists a singular value $v(\la_0)$ such that $f_{\la_0}^n(v(\la_0))=\infty$ for some $n\geq 0$, but this relation does not hold for all $\la$ in a neighborhood of $\la_0$.   The number $n$ is called the {\em order} of the truncated parameter.
	\end{defn}
	If $v(\la_0)$ is an asymptotic value, then $\la_0$ was already named a \emph{virtual cycle parameter}.
	If $v(\lam_0)$ is a critical value, then this means that some critical point for $f_{\lam_0}$ is eventually mapped to a pole; we will say in this case that $\lam_0$ is a \emph{critical prepole parameter}. In particular, truncated parameters are virtual cycle parameters, critical prepole parameters, or both.

	In a similar vein, if $f_{\lam_0}$ has a cycle of multiplier that is a root of unity  we will refer to $\lam_0$ as a \emph{parabolic} parameter; if the multiplier does not remain a root of unity under suitable perturbation of $\lam_0$, then we call it \emph{nonpersistent}.

The density of Misiurewicz parameters in the boundary of the Mandelbrot set, or the density of other dynamically  (and algebraically) defined parameters  are well known consequences of Montel's theorem, as is the accumulation of centers of hyperbolic components on every boundary point. In our setting of natural families of meromorphic maps, we will show that similar results still hold, although the arguments are  more involved and require the use of Theorems A and B. As an example, density of  parabolic parameters (Corollary \ref{cor:density parabolic parameters}) is not immediate, but rather  follows from density of truncated parameters (Proposition \ref{prop:density_of_sp}) and approximation of such parameters by parabolic parameters (Theorem D) and by parameters with superattracting cycles (Proposition~\ref{prop:acc_centers}).

Before we can prove the useful Proposition \ref{prop:density_of_sp}, we first need to deal with the exceptional case of maps without non-omitted poles. The following proposition shows that when the family has no non-omitted poles then everything works as in the rational or entire case. 

\begin{prop}\label{prop:excep}
	Let $\{f_\lam\}_{\lam \in M}$ be a natural family of finite type meromorphic maps.
	\begin{enumerate}[\rm (a)]
	\item If  for all $\lam \in M$, $f_\lam$ does not have any non-omitted pole, then cycles never disappear to $\infty$ in the family $\{f_\lam\}_{\lam \in M}$.
	\item If there exists $\lam_0 \in M$  such that $f_{\lam_0}$ has at least one non-omitted pole, then for all $\lam \in 	M$ outside some proper analytic subset of $M$,  $f_\lam$ has at least one non-omitted pole.
        \end{enumerate}
\end{prop}

\begin{proof}
	To prove (a), observe first that the assumption implies that for all $\lam \in M$, $f_\lam$ has at most one pole, which (if it exists) is also an omitted value, hence an asymptotic value.
	By Lemma \ref{lem:infinitypassive}, either all the maps $f_\lam$ are entire, or they all have exactly one pole, which is also an asymptotic value. By \cite{el}, in the first case cycles never disappear to infinity and we are done. Let us therefore we are in the second case and denote by $v_\lam$ this pole/asymptotic value. 
	
	Assume for a contradiction that a cycle disappears at $\infty$ at some $\lam_0 \in M$. By Theorem A, the map $f_{\lam_0}$ has a virtual cycle which contains either an active asymptotic value or an active critical point. Since $v_\lam$ is the only pole of $f_\lam$, that virtual cycle must be either of the form $\infty, \ldots, \infty$, 
	or of the form $\infty, v_\lam, \infty, \ldots , v_\lam$. It is clear that $v_\lam$ is always passive, and by 
	Lemma \ref{lem:infinitypassive},  $\infty$ is always an asymptotic value, so it is also always passive.
	Therefore, the only remaining possibility is that a critical point collides (non-persistently) with $v_{\lam}$ at $\lam=\lam_0$.
	
	But by the proof of Lemma \ref{lem:infinitypassive}, we have $v_\lam=\psi_\lam(v_{\lam_0})$, and critical points of $f_\lam$
	are of the form $\psi_\lam(c)$, where $c$ are critical points of $f_{\lam_0}$. Therefore, if $v_{\lam_0}$ is a critical point,
	then so is $v_{\lam}$ for every $\lam \in M$, and so it is a passive critical point, contradicting Theorem A.

To show part (b), recall that $f_\lam=\phi_\lam \circ f \circ \psi_\lam^{-1}$, where $f:=f_{\lam_0}$ and $\phi_{\lam_0}=\psi_{\lam_0}=\id$. The set of poles of $f_\lam$ is the set $\psi_\lam \left( f^{-1}(\{\phi_\lam^{-1}(\infty)\})  \right)$.
	Let us denote by $E$ the set (possibly empty) of (Picard) exceptional values of $f$.
	
	We will distinguish two cases: either $\lam \mapsto \phi_\lam^{-1}(\infty)$ is constant or it is not.
	Assume first that it is: $\phi_\lam^{-1}(\infty) \equiv \infty \in \rs$ on $M$ (since $\phi_{\lam_0}(\infty)=\infty$). Then the poles of $f_\lam$ are the 
	$\psi_\lam(a_i)$, $a_i \in f^{-1}(\{\infty\})$, and they move holomorphically over $M$. Let $p$ be the non-omitted pole of 
	$f_{\lam_0}$, and $p_\lam:=\psi_\lam(p)$. Observe that exceptional values of $f_\lam$ (if they exist) are of the form 
	$\phi_\lam(v)$, where $v \in E$. In particular, 
	if $p_\lam \neq \phi_\lam(v)$ for every $v \in E$, then $f_\lam$ has a non-omitted pole.
	Since  $p_{\lam_0} \neq \phi_{\lam_0}(v)$ for every such $v$, the set
	$\bigcup_{v \in E} \{\lam \in M: p_\lam = \phi_\lam(v)\}$ is either empty or an analytic hypersurface of $M$.

	Let us now assume that $\phi_\lam^{-1}(\infty)$ is non-constant. Then for every $\lam \in M$ such that $\phi_\lam^{-1}(\infty) \notin  E$, $f_\lam$ has infinitely many poles, and in particular it has non-omitted poles. Therefore, outside of the
	proper analytic set $\bigcup_{v \in E} \{\lam \in M: \phi_\lam(v)=\infty \}$, $f_\lam$ has infinitely many poles and in particular some non-omitted poles.
\end{proof}

In view of this result, we will focus in what follows on the case of natural families in which there is 
at least one $\lam_*$ such that $f_{\lam_*}$ has at least one non-omitted pole. Indeed, otherwise, it follows from Proposition \ref{prop:excep} that everything works exactly as for the classical case of rational maps.

\begin{prop}[Truncated parameters are dense in the activity locus]\label{prop:density_of_sp}
	Let $\{f_\lam\}_{\in \in M}$ be a natural family of finite type meromorphic maps, and assume that there exists $\lam_* \in M$ 
	such that $f_{\lam_*}$ has at least one non-omitted pole. Assume that an asymptotic (resp. critical) value is active at some $\lam_0 \in M$.
	Then $\lam_0$ can be approximated by virtual cycle parameters (resp. centers) of arbitrarily large order.
\end{prop}

\begin{proof}
By Lemma  \ref{lem:perturbA} and Proposition \ref{prop:excep}, we can perturb $\lam_0$ if necessary to find some $\lam_1$ that is  still in the activity locus, but for which every $f_\lam$ for $\lam$ in a neighborhood of $\lam_1$ has a non-omitted pole. In the rest of the proof, we still denote it by $\lam_0$.

Then, either there is no neighborhood $U$ of $\la_0$ for which $\{f_\la^n(v_\la)\}_n$ is defined  for all $n$ and all $\la\in U$; or for every neighborhood $U$ of $\la_0$ where the family $\{f_\la^n(v_\la)\}_n$ is well defined, it is not normal.

In the first case, $\la_0$ can be approximated by truncated parameters by definition of those. Moreover these truncated parameters must have unbounded orders or otherwise, there exists $N>0$ and a sequence of $\la_k\to \la_0$ such that $f_{\la_k}^N(v_{\la_k})=\infty$. By continuity, $f_{\la_0}^N(v_{\la_0})=\infty$ and by the identity theorem $f_\la^N(v_\la)=\infty$ for all $\la \in U$ (and in fact for all $\la \in M$), which means that $v_\la$ is passive at $\la_0$, a contradiction. 

In the second case, let $p_1(\la)$ and $p_2(\la)$ be two distinct prepoles  varying analytically with $\la$ in $U$. It follows that the family of maps $g_n(\la)=\frac{f_\la^n(v_\la)-p_1(\la)}{p_1(\la)-p_2(\la)}$  is not normal as well, hence it must take the value $0, -1$ or $\infty$ for infinitely many different values of $n$. Since it cannot take the value infinity because the poles are distinct, it follows that it attains $0$ or $1$ infinitely many times, which correspond to truncated parameters $\la\in U$ of order $n+1$ tending to infinity.  

To prove the density of truncated parameters in $\act$ it only remains to see that they themselves belong to the activity locus. But this is straightforward from the definition because if $\la_0$ is a truncated  parameter for $v_\la$, it means that $f_{\la_0}^N(v_{\la_0})=\infty$ for some $N\geq 0$, and the relation is non-persistent. Therefore the family $\{f_\la^n(v_\la)\}_n$ cannot be well defined in any neighborhood of $\la_0$. 
\end{proof}


As a corollary of Proposition~\ref{prop:density_of_sp} and Proposition~\ref{pseudomontel} (the Shooting Lemma) we obtain the density of other dynamically meaningful parameters, namely Misiurewicz and escaping parameters.
\begin{coro}[Density of Misiurewicz and escaping parameters] Under the assumptions of Proposition~\ref{prop:density_of_sp}, \begin{enumerate}
\item The set of Misiurewicz parameters (for which a singular value is eventually and non-persistently mapped to a repelling periodic point) is dense in  $\act$.
\item If there exists an escaping point which depends holomorphically on $\lambda$, the set of parameters for which there is a singular value with orbit converging to infinity is dense (escaping parameters).
\end{enumerate}
\end{coro}

Last before the proof of Theorem D, we see two more approximation results, this time from the stable locus. On the one hand Proposition \ref{prop:acc_centers} will show {\em centers} (i.e. parameters for which a critical point is periodic)  accumulate onto (a subset of) the activity locus (Proposition \ref{prop:acc_centers}). On the other hand,  as long as at least one map in the family has non omitted poles, every active parameter is accumulated  by parameters for which there is an attracting cycle, of periods increasing to infinity. This last statement is Corollary \ref{lem:highattrper}, which follows from Theorem B and Proposition \ref{prop:acc_centers}.

\begin{prop}[Critical prepole parameters are accumulated by centers] 
\label{prop:acc_centers}
Let $(f_\lambda)_{\lambda\in M}$ be a natural family of meromorphic maps of finite type. Let $\lambda_0$ be a  critical prepole parameter of order $n\geq 0$ for the critical value $v_{\lambda_0}$, that is $f^n_{\lambda_0}(v_{\lambda_0})=\infty$. Assume further that $v_{\la_0}$ has a critical preimage $c_{\la_0}$ which is not a Picard exceptional value for $f_{\la_0}$. Then there exists a sequence of parameters $\la_{k} \to \la_0$ as $k\to\infty$,  such that  each $f_{\la_k}$ has a superattracting cycle of period  $n+3$.
\end{prop}

\begin{proof}
Since $f_\la$ is a natural family, critical points can be followed holomorphically with $\la$. Hence, let $c(\la):=\psi_\lam(c_{\lam_0})$, so that $c(\lam)$ is a critical point of $f_\lam$ and $c(\lam_0)=c_{\lam_0}$.

	Up to restricting to a suitable disk in parameter space, we may assume without loss of generality that $M=\D$ and $\lam_0=0$. We let $f:=f_0=f_{\lam_0}$. 
	We distinguish two cases: either there is at least one preimage $a$ of $c_0$ which is neither critical nor singular; or there is none.
	
	\begin{enumerate}
			\item Let us first assume that there is a preimage $a$ of $c_0$ which is neither critical nor singular.
			By the Implicit Function Theorem, there exists a holomorphic map $a(\lam)$ defined near $\lam=0$ 
			such that $f_\lam(a(\lam))=c(\lam)$, where $c(\lam)=\psi_\lam(c_{\lam_0})$. By Proposition 	
				\ref{pseudomontel}, 
			there exists a sequence $\lam_k \to 0$ such that  $f_{\lam_k}^{n+1}(v_{\lam_k})=f_{\lam_k}^{n+2}(c({\lam_k})) =a(\lam_k)$.
			Then $c(\lam_k)$ is a superattracting periodic point of period $n+3$, and we are done with this case.
			\item  Otherwise : since $c_0$ is non-exceptional and there are only finitely many singular values, there are infinitely many critical preimages of $c_0$ which are not singular values, and they must all be critical. Let  us pick 3 of them and let us call them $a_1, a_2, a_3$. The Implicit Function Theorem cannot be applied in this case, but if we let 
			$a_i(\lam):=\psi_\lam(a_i)$, then the $a_i(\lam)$ are critical points of $f_\lam$, and $f_\lam(a_i(\lam))=
			 \phi_\lam(c_{\lam_0}) =: w(\lam)$ is a critical value. 
		 Let us emphasize here that although $ w(0)=c(0)=c_{\lam_0}$,  $w(\lam) \neq c(\lam)$ in 
		 general when $\lam \neq 0$. In particular, $w(\lam)$ is a critical value for $f_\lam$ but is not 
		 necessarily a critical point.
			By assumption, $f^{n+1}(w(\lam_0))=\infty$.
			Let 
			\[
			G(\lam):=f_\lam^{n+2}( w(\lam))\]
		 and 
			\[
			H(\lam):=\frac{G(\lam)-a_1(\lam)}{G(\lam)-a_2(\lam)} \frac{a_3(\lam)-a_2(\lam)}{a_3(\lam)-a_1(\lam)}.
			\] 
			The map $G$ has an isolated singularity at $\lam=0$, which is an essential 
			singularity by the Shooting Lemma (Proposition \ref{pseudomontel}); 
			therefore the same holds for $H$.
			By Picard's theorem, $H$ cannot omit $0,1,\infty$ in any neighborhood of $0$; therefore, we can 
			find a sequence $\lam_k \to 0$ such that $G(\lam_k) = a_{i_k}(\lam_k)$, for $1 \leq i_k \leq 3$. 
			This means that the critical point $a_{i_k}(\lam_k)$ is periodic of period $n+3$, and we are done 
			in this case as well.
		\end{enumerate}	
\end{proof}

\begin{rem}
	The assumption that $c_{\lam_0}$ is non-exceptional is necessary, as shown by the family $f_\la(z)=\tan(\frac{\pi}{4} \la z (z-2i))$. Maps in this family have a unique critical point at $c=i$, which is omitted and hence exceptional. The family is natural since it satisfies trivially the conditions of Theorem \ref{thm:natural} (the critical value is $\lambda$).  Hence the parameter $\la_0=\pi/2$ is a critical prepole parameter since $f_{\pi/2}(i)= \pi/2$ which is a pole.    On the other hand no such map can have a superattracting cycle since the only critical point is omitted.
	\end{rem}

\begin{coro}[Active parameters are accumulated by parameters with attracting cycles]\label{lem:highattrper}
	Let $\{f_\lam\}_{\lam \in M}$ be a natural family of finite type meromorphic maps, and assume that there exists $\lam_* \in M$ 
	such that $f_{\lam_*}$ has at least one non-omitted pole. Assume that a singular value $v(\lam_0)$ is active at some $\lam_0 \in M$.
	Then $\lam_0$ can be approximated by parameters $\lam$ such that $f_\lam$ has attracting cycles of arbitrarily large periods.
\end{coro}

\begin{proof}
	 If there exist  $n \in \N$ and  $v_i(\lam) \in S(f_\lam)$  such that $f_\lam^n(v_i(\lam)) \equiv \infty$
	 let 
	$$N_0:=\max \{n \in \N: f_\lam^n(v_i(\lam)) \equiv \infty, v_i(\lam) \in S(f_\lam)\}.$$
	Otherwise, $N_0=0$.
	Let $U$ be a neighborhood of $\lam_0$ in $M$, and let $N >N_0$. We will prove that there exists $\lam \in U$ such that $f_\lam$ has an attracting cycle of period at least $N$.
	By Proposition \ref{prop:density_of_sp}, there exists $\lam_1 \in U$ such that $f_{\lam_1}^{n}(v(\lam_1))=\infty$  for some $n\geq N$, 
	and $f_{\lam}^{n}(v(\lam))\not\equiv\infty$ on $M$.
	
	By Theorem B and Proposition \ref{prop:acc_centers}, there is only one  remaining case to consider: the case where $v(\lam_1)$ is a critical value, with $v(\lam_1)=f_{\lam_1}(c(\lam_1))$ and $c(\lam_1)$ is a critical point which is  an exceptional value for $f_{\lam_1}$. 
	
	Let ${w(\la)}:=\phi_\lam(c(\lam_1))$. Since $c(\lam_1)$ is an exceptional value, it is also an asymptotic value; therefore ${ w(\la)}$ is an asymptotic value for $f_\lam$ (see Section \ref{sec:natural}). By our choice of $N$, we have $f_{\lam_1}^{{n+1}}({ w}(\lam_1))=\infty$ but  $f_{\lam}^{{n+1}}({w}(\lam))\not\equiv\infty$. 
	We can then apply Theorem B, to find $\lam_2 \in U$ such that $f_{\lam_2}$ has an attracting cycle of period ${n}+2$.
	This concludes the proof.

\end{proof}

We are now ready to prove Theorem D. Since in the assumptions there is a cycle disappearing at infinity, in view of Theorem A there is at least one function in the family which has  at least one non omitted pole.

\begin{proof}[\bf Proof of Theorem D]

We argue by contradiction, and so we assume that there is a neighborhood $V$ of $\lam_0$ in parameter space such that 
for all $\lam \in V$, $f_\lam$ has no non-persistent parabolic cycles of period up to $n$.
Recall the notations of Section \ref{sec:exiting}: $P_n:=\{(\lam,z) \in M \times \C : z=f_\lam^n(z)\}$, 
and $\pi_1: P_n \to M$ is the projection on the first coordinate.
The idea of the proof is the following: first we reduce to the case when $V$ is a one-dimensional disk centered at $\lam_0$.
Then, we prove that $\pi_1: P_n \to V$ restricts to a finite degree branched cover. This allows us (up to passing to a finite branched cover) to construct a single-valued parametrization of the disappearing periodic point, with only a pole at $\lam_0$. 
Using this parametrization, we show that cycles disappear at infinity from any directions in dynamical space, which finally leads to a contradiction.

\medspace
Let us make this idea precise.
 By assumption, there exists a curve $\tilde \gamma \subset P_n$ such that $\limt \tilde \gamma(t)=\infty$ and $\limt \pi_1 \circ \tilde \gamma(t)=\lam_0$, and moreover, if $\tilde \gamma(t)=:(\lam(t), z(t))$, then $z(t)$ is an attracting periodic point for $f_{\lam(t)}$, of period dividing $n$.
Let $S$ denote the irreducible component of $P_n \cap (V \times \C)$ containing $\tilde \gamma$. We will start by showing that up to restricting $V$, we can assume that $\pi_1^{-1}(\{\lam_0\}) \cap S=\emptyset$.

First, observe that by our assumption on the lack of parabolic cycles, for every $(\lam,z) \in S$, the point $z$ is an attracting periodic point for $f_\lam$. Indeed, by assumption, the multiplier map $\rho: (\lam,z) \mapsto (f_\lam^n)'(z)$ satisfies $\rho(S) \subset \C \setminus S^1$, and $\rho(S) \cap \D = \emptyset$. Since $S$ is a connected Riemann surface, we therefore have $\rho(S) \subset \D$.   But since the maps $f_\lam$ are of finite type, there are only finitely many such points; in other words, $\pi_1: S \to V$ has finite degree. Let $z_0, \ldots, z_m$ denote the elements of ${\pi_1}^{-1}(\{\lam_0\}) \cap S$, if it is not empty.
Let $\epsilon>0$ be small enough that $z_0, \ldots, z_m$ all move holomorphically over $\mathbb{B}(\lam_0,\epsilon)$.
Then the points $(\lam_0, z_i)$ ($0 \leq i \leq m$) are not in the same connected component of $P_n \cap (\B(\lam_0,\epsilon) \times \C)$ as the curve $\tilde \gamma$. This means that up to replacing $V$ by $\B(\lam_0,\epsilon)$, we can indeed assume 
that $\pi_1^{-1}(\{\lam_0\}) \cap S= \emptyset$, which we do from now on. 

In particular, for \emph{every} curve $\tilde \gamma(t)=(\lam(t), z(t)) \subset S$ such that $\pi_1 \circ \tilde \gamma(t)=\lam(t) \to \lam_0$, we must have $z(t) \to \infty$, since any finite accumulation point would otherwise be an element of $\pi_1^{-1}(\{\lam_0\}) \cap S$. Let $D$ be an embedded holomorphic disk passing through $\lam_0$ in $V \subset M$, and such that for all 
$\lam \in D^*:=D \setminus \{\lam_0\}$, $f_\lam$ has no non-persistent singular relation of the form $f_\lam^k(v(\lam))=\infty$, for $0 \leq k \leq n$ and $v(\lam)$ a singular value, which exists since such parameters form a discrete set.  By Theorem A, it follows that no cycle of period up to $n$ can exit the domain at any $\lam \in D^*$. Therefore, by our assumption on the lack of parabolic cycles, the map $\pi_1: S_0 \to D^*$ is a covering map, where $S_0$ is any irreducible component of $S \cap (D^* \times \C)$.
Moreover, we have already proved that $\pi_1: S_0 \to D^*$ has finite degree, so there exists a conformal isomorphism $\Phi=(\Phi_1,\Phi_2): \D^* \to S_0$, such that $\Phi_1=\pi_1 \circ \Phi: \D^* \to D^*$ is a finite degree cover. 
Therefore, $\Phi_1$ extends holomorphically to a finite degree branched cover $\Phi_1: \D \to D$, with $\Phi_1(0)=\lam_0$. This ends the first part of the proof of reducing to a one dimensional setting.

Now chose any curve $t \mapsto u(t)$ in $\D^*$ such that $u(t) \to 0$, so that $\Phi_1(u(t)) \to \lam_0$. Then by a previous observation, $\Phi_2(u(t)) \to \infty$, since $\Phi(u(t)) \in S_0$ and $\pi_1\circ \Phi(u(t)) \to \lam_0$. (In particular, this means that $\Phi_2$ has a pole at $u=0$).

In order to lighten the notations, let $z_t:=\Phi_2(u(t))$,  $f_t:=f_{\Phi_1(u(t))}$, $\phi_t:=\phi_{\Phi_1(u(t))}$ and $\psi_t:=\psi_{\Phi_1(u(t))}$, where $f_\lam=\phi_\lam \circ f \circ \psi_\lam^{-1}$.

Since $f_t( z_t) = \phi_t \circ f \circ \psi_t^{-1} ( z_t)  \to v \in \chat$
(the next point in the virtual cycle), we have 
\[
f \circ \psi_{t}^{-1}(z_t) = \phi_{t}^{-1}(v+o(1))=v+o(1),
\]
using the fact that $\phi_0=\id$. 
Therefore  $\psi_{t}^{-1} ( z_t)) \in T$,
where $T$ is a tract of $f$ above the asymptotic value $v$. Let $G(u):=\psi_{\Phi_1(u)}^{-1}(z(\Phi_1(u))): \D \to \chat$.  Then $G$ is open by Lemma \ref{lem:Gopendiscrete}, and $G(\D)$ is a neighborhood of $\infty$. 
But we also proved that $G(\D)$ is contained in $T$ a tract of $f$  since $u(t)$ was arbitrary, which is absurd.

\medskip

\end{proof}

Theorem D together with Propositions \ref{prop:density_of_sp} and \ref{prop:acc_centers} implies density of parabolic parameters.   Notice that for this corollary it is not necessary to assume the existence of a function in the family which has  non omitted poles.
 \begin{coro}\label{cor:density parabolic parameters}
 Let $\{f_\lambda\}_{\lambda\in M}$ be a natural family of finite type meromorphic maps. Then parabolic parameters are dense in the activity locus.
\end{coro}

\begin{proof}
 Let $\la_0$ belong to the activity locus and let $V$ be an arbitrary small neighborhood of $\la_0$. Our goal is to produce at least one parabolic cycle in $V$.

By Lemma \ref{lem:highattrper}, for every $k \in \N^*$ large enough, we can find $\lam_k \in V$ such that   
$f_{\lam_k}$ has an attracting cycle of period at least $k$. Each of these attracting cycles must capture at least one singular value.
But the number of singular values is finite and constant throughout $V$. Therefore, only finitely many of these attracting cycles can be followed holomorphically over $V$ and remain persistently attracting on  $V$.

This implies either the existence of parabolic parameters or the existence of a parameter $\lam$ at which an attracting cycle disappears to infinity. In the first case we are done, and in the second as well by Theorem D.

\end{proof}


\subsection{$ \JJ-$stable parameters: proof of Theorem E, Corollary F and Corollary G} \label{sec:MSS}

From our results above we can now prove  Theorem E. 

\begin{proof}[\bf Proof of Theorem E]
	If for all $\lam \in M$, $f_\lam$ never has any non-omitted pole, 
	then the result is trivial by the classical theory and 
	Proposition \ref{prop:excep}. We therefore assume that the assumptions of Proposition \ref{prop:density_of_sp} are satisfied.

	We will prove that $(1) \Leftrightarrow (2)$, $(1)  \Rightarrow (3)  \Rightarrow (2) \Leftrightarrow (5) \Rightarrow (4) \Rightarrow (3)$.

	\begin{itemize}
		\item $(1) \Rightarrow (2)$:  This implication mostly follows the arguments in \cite{mss}.  	If the Julia set moves holomorphically over $U$, there is $\la_0\in U$ and a holomorphic motion $H:U \times J_{\la_0} \to \hat{\C}$ preserving the dynamics. Hence $H_\la(J_{\la_0})=J_\la$ for all $\la\in U$ and 
		\[
		H_\la (f_{\la_0}(z)) = f_\la (H_\la(z))
		\]
	for all $z\in J_{\la_0}$. This means that $H_\la$ maps  critical points (resp. values) of $f_{\la_0}$ in $J_{\la_0}$ to critical points (resp. values) of $f_\la$  in $J_\la$ (see e.g. \cite{mcmullen2} for details). Likewise $H_\la$ maps  asymptotic values of $f_{\la_0}$  in $J_{\la_0}$ to asymptotic values of $f_\lambda$  in $J_{\la}$, since the latter are locally omitted values.  

		Hence singular values and their full orbits  in the Julia set can be followed by the conjugacy $H_\la$. 
		Since $f_{\lambda_0}$ has finitely many singular values, the union of their (forward) orbits is a countable set; but the Julia set is perfect and uncountable, hence we can consider   three points $z_1, z_2$ and $z_3$ in $J_{\la_0}$ which are disjoint from the forward orbits of the singular values of $f_{\la_0}$. Consequently, by the injectivity of the  holomorphic motion,  for all $\la\in U$, $H_\lam(z_i)$, $i=1,2,3$, is disjoint from the forward orbits of the singular values of of $f_\la$. By  Montel's Theorem  it follows that the forward orbits of the singular values form normal families, and hence every singular value is passive in $U$. On the other hand, if a singular orbit lies in the Fatou set of $f_{\la_0}$ then it must remain in the Fatou set of $f_\la$ for every $\la\in U$. The orbit then misses all points in the Julia set and the same argument applies.
		\item $(2) \Rightarrow (1)$: Assume that the Julia set does \emph{not} move holomorphically over $U$. Then by  Proposition \ref{lem:singular values iff no holo motion}, either two periodic points in the Julia set collide, or one periodic cycle in the Julia set exits the domain. In the first case, 
		this means that there exists $\lam_0 \in U$ with a non-persistent parabolic periodic point: there exists $z_{\lam_0} \in \C$ such that $f_{\lam_0}^n(z_{\lam_0})=z_{\lam_0}$, $(f_{\lam_0}^n)'(z_{\lam_0})=1$,
		and $\lam \mapsto (f_{\lam_0}^n)'(z_{\lam_0})$ is non-constant on $U$.
		Then its parabolic basin must contain at least one singular value $v_{\lam_0}$, and therefore be active. 
		In the second case, a cycle exits the domain at $\lam_0 \in U$, and by Theorem A, $f_{\lam_0}$ has either an active critical value or an active asymptotic value.
		\item $(1) \Rightarrow (3)$: Assume that the Julia set moves holomorphically over $U$, and let $H_{\la}$ be the conjugating holomorphic motion as above. Then $H_{\la}$ maps repelling periodic points of $f_{\lam_0}$  to repelling periodic points of $f_\lam$ in $J(f_\lam)$ of the same period. Let $N$ be the maximal period of non-repelling cycles for $f_{\lam_0}$ (which is finite by Fatou-Shishikura's inequality \cite{el}); then for all $\lam \in U$, cycles of period more than $N$ must be repelling, which implies that attracting cycles have period at most $N$.
		\item  $(3) \Rightarrow  (2)$: 
		This follows directly from Lemma \ref{lem:highattrper}.
\item $(5)\Rightarrow (4)$: If there are no non-persistent parabolic parameters in $U$, there cannot be any  $\lam \in U$ such that $f_\lam$ has a non-persistent virtual cycle, and hence no cycle can exit the domain for a parameter in $U$ by Theorem A. Hence no attracting cycle can disappear to infinity nor change into a repelling cycle, which implies that all attracting cycles remain attracting throughout $U$. 
\item $(4) \Rightarrow (3)$ This implication is obvious.
	\end{itemize}
\end{proof}
	
We can therefore define the bifurcation locus of the family $M$, $\bif(M)$, as the set of $\lam \in M$ for which the above conditions are not satisfied in any neighborhood of $\lam$.
With this notation, Corollary F states: 

\begin{corF}
	Let $\{f_\lam\}_{\lam \in M}$ be a natural family of finite type meromorphic maps. Then $\mathring{\bif(M)}=\emptyset$.
\end{corF}

\begin{proof}
By Theorem E, parabolic parameters are dense in the bifurcation locus. Hence whenever a singular value is active at say $\lambda_0$, one may perturb $\lambda_0$ to make it into a non-persistent parabolic parameter  $\lambda_1$ whose parabolic basin attracts one of the singular values, say $v^1_\lambda$. The parameter  $\lambda_1$ can be perturbed to $\lambda_1'$ for which  $v^1_\lambda$ is attracted to an attracting cycle, persistent in a neighborhood of $\lambda_1'$. If this parameter is not in the bifurcation locus, we are done; otherwise, there is another singular value which is active, and we can perturb it to a new parabolic parameter $\lambda_2$ while keeping $v^1_\lambda$ to be attracted to its attracting cycle, and repeat the reasoning. Since there are only finitely many singular values, this proves that arbitrarily close 	to any $\lam_0 \in \bif(M)$ we may find $\lam_*$ such that all singular values of $f_{\lam_*}$ are passive (because each is attracted to an attracting cycle), and therefore $\lam_* \notin \bif(M)$.
\end{proof}

Finally, we prove Corollary G.

\begin{proof}[Proof of Corollary G]
	Let $v(\lam_0)$ be a singular value of $f_{\lam_0}$, and let $m$ be the period of the attracting cycle it converges to.
	Let $h_n(\lam):=f_\lam^n(v(\lam))$: then $h_{km}(\lam_0)$ converges (as $k \to +\infty$) to an attracting periodic point, which we denote by $h(\lam_0)$. Moreover, this remains true for all $\lam$ in a suitably small neighborhood of $\lam_0$: the sequence $h_{km}$ converges to $h$ near $\lam_0$, where $h(\lam)$ is an attracting periodic point of period $m$.
	
	By Theorem E, all singular values are passive on $U$, therefore the sequence $h_{km}$ admits subsequences that converge locally uniformly on $U$; and any such limit must coincide with $h$ on a neighborhood of $\lam_0$. Therefore, the sequence $h_{km}$ in fact converges on $U$ to a map which we still denote by $h$. Since $f_\lam^m(h(\lam))-h(\lam)=0$ on a neighborhood of $\lam_0$, 
	we have that $h(\lam)$ is a periodic point of period (dividing) $m$ for all $\lam \in U$.
	Moreover, it is clear that this periodic point cannot be repelling: so the multiplier map 
	$\rho(\lam):=(f_\lam^m)'(h(\lam))$ takes values in $\overline{\D}$.
	By the openness theorem, either $\rho$ is non-constant and open, or $\rho$ is constant equal to $\rho(\lam_0) \in \D$. In either case, $\rho(\lam) \in \D$ for every $\lam \in U$, which proves that $v(\lam)$ remains captured by an attracting cycle for all $\lam \in U$. 
	
	Since this holds for any singular value, Corollary G is proved.
\end{proof}


\section{A simple example with a dense bifurcation locus}
\label{bifdense}

In this section we construct a natural family of transcendental entire maps of the form $f_
\lambda=f+\lambda$ with $\lambda\in\C$, and  $f$ chosen so that $f_{\lambda_n}$ has a non-persistent parabolic point for all  $\lambda_n$ in  a countable dense subset of $\C$. It will follow that the bifurcation locus for this family has nonempty interior. With this construction we have no control on the set of singular values for $f$, and consequently, for $f_\lambda$.

\begin{prop}[Bifurcation locus with nonempty interior]
There exists an entire map $f$ such that the natural  family $\{f_\lambda=f+\lambda\}_{\lam \in \C}$  is not $\JJ$-stable on any open subset of $\C$.
\end{prop}
\begin{proof}
Let $\{\lam_n: n\in\N\}$  be a countable  dense subset of  the complex plane. 
By  classical results (see e.g. Theorem \cite[Theorem 15.13]{rudin}),  there exists en entire map $f$ such that for all $n \in \N$, 
	\begin{enumerate}
		\item $f(n)=n-\lam_n$
		\item $f'(n)=-1$
		\item $f''(n) = 1$.
	\end{enumerate}	
It follows that  for all $n \in \N$,   $f_{\lam_n}(n)=n$,  $f_{\lam_n}'(n)=-1$, and $f_{\lam_n}''(n) = 1$.
In other words, each $f_{\lambda_n}$ has a  parabolic fixed point of multiplier $-1$ at the positive integer $n$.   We now show that such fixed  points are not  persistently parabolic by studying their multiplier map. 


By the Implicit Function Theorem, for each $n \in \N$  there exists a holomorphic map $\lam \mapsto z_n(\lam)$ defined in a neighborhood $U_n$ of $\lam_n$, such that $f_{\lam}(z_n(\lam)) = z_n(\lam)$ and $z_n(\lam_n)=n$. 

Differentiating the relation
$$
f_{\lambda}(z_n(\lambda))-z_n(\lambda)\equiv 0 \text{ on $U_n$,}
$$
we find $\frac{d}{d\lam}_{|\lam=\lam_n} z_n(\lam) =  \frac{1}{2}$.                 

We now consider the multiplier map 
 defined as  $\rho_n(\lam):=f_{\lam_n}'(z_n(\lam))$. We want to show that it is not constant on $U_n$. We have
  $$
  \frac{d}{d\lam}_{|\lam=\lam_n} \rho_n(\lam) 
  = f_{\lam_n}''(z_n(\lam_n)) \frac{d}{d\lam}_{|\lam=\lam_n} z_n(\lam) + \frac{d}{d\lam}_{|\lam=\lam_n} f_\lam(z_n) = 1 \cdot \frac{1}{2} + 1 = \frac{3}{2} \neq 0.
  $$
 Therefore, the map $\rho_n$ is non-constant, and 
	so the fixed point $z_n(\lam)$ is non-persistently neutral for $f_{\lam}$. 
We now show that this implies that the family $\{f_\lam\}_{\lam \in \C}$ cannot be $\JJ$-stable in any neighborhood of $\lam_n$.

	Indeed, assume for a contradiction that $\{f_\lam\}_{\lam \in \C}$ is $\JJ$-stable on some disk $D$ centered at $\lambda_n$, and let $h_{\lam}$ denote the corresponding holomorphic motion of $\JJ(f_\lam)$,  which must respect the dynamics. 
	Since $z_n(\lam_n)=n$ is a parabolic fixed point for $\lambda_n$, it lies in $\JJ(f_{\lam_n})$. Then the motion  $h_{\lam}(n)$ must coincide with $z_n(\lam)$ for $\lam \in D \cap U_n$, by unicity in the Implicit Function Theorem. In particular,  $z_n(\lam)=h_{\lam}(n)$ must be in  $\JJ(f_{\lam})$  for all $\lambda\in U_n \cap D$. On the other hand, since $\rho_n$ is non-constant,   there exists $\lam_*$ arbitrarily close to $\lambda_n$ such that  $\rho_n(\lam_*) \in \D$;  hence $z(\lambda_*) $ is attracting and must belong to the Fatou set of $f_{\lambda_*}$, a contradiction.	
	We are done since $\{\lam_n: n \in \N\}$ is dense in $\C$.
\end{proof}

\bibliographystyle{amsalpha}	
\bibliography{BifMeroABF}	
\end{document}